\documentclass[a4paper,11pt]{amsart}
\usepackage[top=2.2cm,left=2.8cm,right=2.8cm,bottom=2.8cm]{geometry}
\usepackage[foot]{amsaddr}

\usepackage{hyperref}
\hypersetup{
    colorlinks,
    citecolor=green,
    filecolor=black,
    linkcolor=blue,
    urlcolor=blue
}
\hypersetup{linktocpage}

\usepackage{cite}
\usepackage{graphicx}
\usepackage{amsbsy}
\usepackage{latexsym}
\usepackage{amsfonts}
\usepackage{amssymb}
\usepackage{upgreek}
\usepackage[usenames]{color}
\usepackage{amsmath,amsthm}
\usepackage{enumerate}
\usepackage{mathrsfs}
\usepackage{stmaryrd}
\usepackage{dsfont}
\usepackage{varwidth}
\usepackage{pdfpages}
\usepackage{tikz}

\newcommand{\eps}{\varepsilon}

\DeclareMathOperator{\rank}{rank}

\DeclareMathOperator*{\argmin}{arg\,min}

\providecommand{\abs}[1]{\lvert#1\rvert}

\providecommand{\norm}[1]{\lVert#1\rVert}

\providecommand{\biggnorm}[1]{\biggl\lVert#1\biggr\rVert}

\renewcommand{\Re}{\operatorname{Re}}

\newtheorem{theorem}{Theorem}
\newtheorem{lemma}[theorem]{Lemma}
\newtheorem{prop}[theorem]{Proposition}

\theoremstyle{definition}

\newtheorem{example}[theorem]{Example}

\theoremstyle{remark}
\newtheorem{remark}[theorem]{Remark}

\numberwithin{equation}{section}

\theoremstyle{plain}

\newcommand{\cW}{\mathcal{W}}

\newcommand{\Chi}{\raise .3ex
\hbox{\large $\chi$}}

\newcommand{\R}{\mathbb{R}}

\newcommand{\N}{\mathbb{N}}
\newcommand{\Z}{\mathbb{Z}}
\newcommand{\C}{\mathbb{C}}

\usepackage[section]{algorithm}
\usepackage{algorithmicx}
\usepackage{algpseudocode}

\usepackage{todonotes}

\title{Iterative thresholding low-rank time integration}
\author{Markus Bachmayr, Matthieu Dolbeault, Polina Sachsenmaier}
\address{Institut f\"ur Geometrie und Praktische Mathematik, RWTH Aachen University, Templergraben 55, 52062 Aachen, Germany}
\email{\{bachmayr,dolbeault,sachsenmaier\}@igpm.rwth-aachen.de}

\thanks{M.B. acknowledges funding by Deutsche Forschungsgemeinschaft (DFG, German Research Foundation) -- project number 442047500/SFB 1481 \emph{Sparsity and Singular Structures} and by the European Union (ERC, COCOA, 101170147). M.D.\ was supported by Deutsche Forschungsgemeinschaft (DFG, German Research Foundation) -- project number 442047500/SFB 1481 \emph{Sparsity and Singular Structures}. P.S.\ was supported by Deutsche Forschungsgemeinschaft (DFG, German Research Foundation) -- project number 320021702/GRK2326 \emph{Energy, Entropy, and Dissipative Dynamics}.}

\begin{document}

\maketitle

\begin{abstract}
We develop time integration methods in low-rank representation that can adaptively adjust approximation ranks to achieve a prescribed accuracy, while ensuring that these ranks remain proportional to the corresponding best approximation ranks. Our approach relies on an iterative scheme
combined with soft thresholding of the iterates. 
A model case of a time-dependent Schr\"odinger equation with low-rank matrix approximation is analyzed in detail, and the required modifications for second-order parabolic problems are described. Numerical tests illustrate the results for both cases.

\smallskip

\end{abstract}

\section{Introduction}

Standard numerical methods for solving partial differential equations (PDEs), when applied to problems on domains of large dimension, typically suffer from the \emph{curse of dimensionality}: namely, their computational cost, as well as the space required to store the solution, scale exponentially in the dimension, making the methods impractical even at low resolution. In many cases of interest, however, such limitations can be overcome by appropriate compressed representations of approximate solutions, in particular by low-rank tensor representations. 

Here we consider both dispersive and parabolic PDEs. The former arise as high-dimensional time-dependent problems in quantum mechanics, with the linear Schrö\-din\-ger equation
\begin{equation}
\label{eq_Schrodinger}
i\partial_t u=-\Delta u+Vu
\end{equation}
 posed over $\R^{3N}$ when describing a system of $N$ particles in $\R^3$ subject to the potential function $V$. When working with low-rank approximations, such problems are commonly formulated in the framework of second quantization in terms of occupation numbers of orbitals. 
Another motivating class of examples, this time of parabolic type, are Fokker-Planck equations of the form
\begin{equation}
\label{eq_FokkerPlanck}
\partial_t u={\rm div}(D\nabla u+u P),
\end{equation}
which describe probability densities $u$ for trajectories of stochastic processes, as in the overdamped Langevin equation in molecular dynamics. Here, the dimension of the underlying space corresponds to the total number of degrees of freedom of the considered molecular system, the matrix $D$ is a diffusion coefficient, and the vector field $P$ corresponds to the acting forces.

Approximations based on low-rank tensor representations have a long history for both of these classes of problems. In many cases of interest, such representations can avoid the curse of dimensionality by a reduction to sums of tensor products of lower-dimensional factors. However, their efficiency depends crucially on the sizes of certain rank parameters in the tensor formats, which reduce to matrix ranks in the case of tensor order two. In this regard, a central issue in working with such low-rank representations is that they do not form linear spaces: under vector space operations and other mappings, their rank parameters generally increase. For reasonable computational costs, methods operating on low-rank tensor representations thus need to ensure that rank parameters remain under control.

For time-dependent PDEs, several approaches for low-rank approximation exist that use different mechanisms for controlling ranks. 
These range from methods that keep the ranks fixed, such as dynamical low-rank approximation, which may lead to uncontrolled errors, to methods that approximate standard time stepping schemes up to any desired accuracy, but may instead lead to unnecessarily large ranks. While for methods based on space-time variational formulations, first results on how to achieve controlled ranks for given guaranteed errors exist, the availability of suitable such formulations can be a demanding condition.

Our main result in this work is a method for time integration that does not require a space-time variational formulation, but still yields provably convergent approximations while at the same time ensuring control of the ranks of the produced approximations and all intermediate quantities. More precisely, we show that these ranks remain comparable to their natural benchmark quantities, that is, to the best approximation ranks for the sought solutions at the achieved accuracy.

We focus here on representations in terms of matrices in low-rank representation, which involve fewer technicalities than higher-order tensors. However, for our purposes, they exhibit all main difficulties of the higher-order case, since the target tensor formats for higher-dimensional problems, such as tensor trains and hierarchical tensors, are themselves based on matrix ranks of certain matricizations (see, for example, the monograph \cite{Hackbusch:12tensorspaces} and the surveys \cite{GKT13,Bachmayr23}). In particular, the standard methods for error-controlled rank reduction that are available for these higher-order tensor formats reduce to truncation or thresholding of the singular value decomposition in the matrix case.
The schemes proposed here can thus be applied to general hierarchical tensor representations without essential modifications.

\subsection{Rationale and main result}

We consider here time-dependent problems in two spatial variables on a product domain $\Omega = \Omega_1 \times \Omega_2$ with $\Omega_i \subseteq \R^{d_i}$ for $i=1,2$, using low-rank approximations to obtain a reduction to operations on lower-dimensional factors; this is also the main aim in more general higher-dimensional problems. 

The low-rank approximations that we aim to construct for solutions $u \in C([0,T]; L^2(\Omega))$ are of the basic form
\begin{equation}\label{eq:lrapprox}
   u(t, x_1, x_2) \approx \sum_{k=1}^{r(t)} u^{(1)}_k(t, x_1)\, u^{(2)}_k(t, x_2),\quad\text{for}\quad t\in [0,T], \quad x_1 \in \Omega_1, \quad x_2 \in \Omega_2,
\end{equation}
where $u^{(i)}_k  \in L^2(\Omega_i)$ for $i=1,2$ and $k = 1,\ldots, r(t)$. Here a main difficulty is to adapt the rank parameter $r(t)$ to a prescribed error.

For each fixed $t$, low-rank best approximations of $(x_1,x_2)\mapsto u(t, x_1, x_2)$ as an element of $L^2(\Omega) = L^2(\Omega_1)\otimes L^2(\Omega_2)$ for each given rank can be obtained by singular value decomposition, where the error of low-rank approximation is given by the $\ell^2$ norm of the sequence of omitted singular values. For each given $\varepsilon > 0$, these tail norms determine a minimum rank $\hat r_\varepsilon(t)$ such that the best approximation of this rank has $L^2$-error at most $\varepsilon$.

Our main objective is to obtain, for any prescribed accuracy $\varepsilon$, low-rank approximations as in \eqref{eq:lrapprox} with quasi-optimal ranks $r(t)$ -- that is, we aim at a method that produces approximations with $r(t)$ comparable, up to constants, to the ranks $\hat r_\varepsilon(t)$ of \emph{best} low-rank approximations of the same accuracy.
At the same time, the ranks arising in the computational method should also remain controlled in terms of the ranks of the final result. 

This requires a balance between error accumulation and applying sufficient rank truncations.
Such a balance is difficult to achieve in step-truncation methods based on classical time stepping schemes: to ensure sufficiently small cumulative errors, also the rank truncation tolerances for each time step need to tighten accordingly, and a comparison to best approximation ranks is difficult to achieve -- indeed, no result of this type appears to be known for such methods. 

In view of the extension of our approach to higher dimensions, what we mainly aim for is a reduction to operations on lower-dimensional factors in low-rank representations. At the same time, the methods that we consider here do not rely on direct manipulations of lower-dimensional components of tensor representations and are thus directly generalizable to various low-rank formats for higher-order tensors. Especially in the high-dimensional case, the overall computational costs are then driven mainly by the ranks of tensor representations. 

To obtain ranks that are appropriate for the achieved approximation error, we use iterative refinement of approximate solutions in basis representations in time, applied on subintervals of length $h$ of the given total time interval $[0,T]$. Here, unlike standard time stepping methods, we do not necessarily let $h$ tend to zero to enforce convergence, but rather consider fixed (or slowly decreasing) $h$ in combination with a sufficiently refined temporal basis in each subinterval. Unlike space-time variational formulations, we rely only on standard integral formulations of evolution problems with uniform approximation in time.

To ensure that $h$ can be chosen to be relatively large and remains in particular independent of the spatial discretization, we use techniques similar to exponential integrators \cite{HO:10}, but with particular adaptations to the low-rank setting.
The basic idea was presented by Lawson in \cite{Lawson67} for Runge-Kutta methods and later used in exponential integrators \cite{KT05}, see also \cite{HLO20}. 

For the Laplacian $\Delta$ on $\Omega$, we have $\Delta = \Delta_1 \otimes I + I \otimes  \Delta_2$, where $\Delta_i$ is the Laplacian on $\Omega_i$ for $i=1,2$. For $u_i \in L_2(\Omega_i)$ and any $z \in \C$ with $\Re z \geq 0$,
\begin{equation}\label{eq:expfactors}
   e^{z \Delta} (u_1 \otimes u_2) = e^{z \Delta_1} u_1 \otimes e^{z \Delta_2} u_2.
\end{equation}
As a consequence, the action of $e^{z\Delta}$ leaves ranks unchanged and thus can directly be reduced to lower-dimensional operations.
Our starting point to make use of this fact is Duhamel's formula, which in the case of the time-dependent Schr\"odinger equation \eqref{eq_Schrodinger} with initial data $u_0$ reads
\[
u(t) = e^{i t \Delta}u_0 - i \int_{0}^{t} e^{i(t - s)\Delta} V u(s) \, ds.
\]
For this class of problems, we additionally use the (rank-preserving) transition to the \textit{twisted variable} $v(t) = e^{-it \Delta}u(t)$ to arrive at the formulation
\begin{equation}\label{eq:twisted_intro}
		v(t)  =  u_0 - i \int_{0}^{t} e^{-i s \Delta} V e^{i s \Delta} v(s) \, ds  \,.
\end{equation}
We now combine this formulation with a choice of temporal basis functions, where we use polynomials of degree $J$ on each subinterval; although other choices could be of interest, we focus on this particular case in this work. 

Under appropriate assumptions on $V$, \eqref{eq:twisted_intro} can serve as the basis for fixed-point iterations for refining approximations of $v$ on each subinterval. These converge for subinterval sizes $h$ that depend only on $V$, but not on the spatial discretizations.
The iterations are performed by considering \eqref{eq:twisted_intro} in suitable sets of collocation points. As a consequence, the corresponding fixed points are steps of high-order implicit Runge-Kutta schemes on each subinterval. 
For the particular choice of Gauss-Legendre points, we obtain the Gauss-type methods, which for the unitary evolution of \eqref{eq_Schrodinger} have favorable norm and energy preservation properties.

We arrive at an iterative scheme that successively solves on each subinterval using Picard iteration or a variant known as \emph{spectral deferred correction}, where the approximate solutions in each collocation point are stored in low-rank form. 
The corresponding ranks increase in each iteration by a factor depending on $J$ and representation rank of $V$. We subsequently reduce them in each step by a \emph{soft thresholding} of singular values, which is a non-expansive operation in $L^2(\Omega)$, with a new technique for controlling the thresholds inspired by \cite{BS17}.

In our main result, Theorem~\ref{thm:main}, we show that with the strategy for adjusting error tolerances and thresholding parameters that we propose for the numerical scheme, we achieve a guaranteed error 
\[  \varepsilon  = C_{J,V}   h^{J+1} \exp( c_{J,V} T )   \]
for some positive constants $C_{J,V}$ and $c_{J,V}$, and obtain quasi-optimal bounds on the ranks $r(t)$ of all computed iterates. That is, if for an $s>0$,  
\[ \max_t \hat r_\eta(t) \lesssim \eta^{-\frac1s}  \]
with the best approximation rank for $\eta >0$ as defined above, then also
\[
  \max_t r(t) \lesssim h^{-3} \bigl(1 + J \rank(V)\bigr)\eta^{-\frac1s},
\]
where $\rank(V)$ is the rank of the function $V$ (for simplicity, we here assume that the potential is sufficiently smooth and of bounded ranks). In other words, the largest ranks produced by the numerical scheme remain bounded up to a constant by the ranks of the result of applying the fixed point mapping to a low-rank best approximation of the solution.
Since the constant depends also on $h^{-3}$, the subinterval size should not be chosen too small, and we rather aim to enforce convergence by refining the approximation in each subinterval, that is, by increasing $J$. In the more favorable case of exponential-type decrease of singular values, 
\[ \max_t \hat r_\eta(t) \lesssim  \bigl( 1+  \abs{\log \eta}  \bigr)^\beta  \]
with $\beta>0$, then accordingly
\[
  \max_t r(t) \lesssim h^{-3} \bigl(1 + J \rank(V)\bigr) \bigl( 1 + \abs{\log \eta} \bigr)^\beta.
\]

We focus on the case of time-dependent Schr\"odinger equations, for which we give a detailed analysis, but also describe the modifications required for the case of second-order parabolic problems motivated by \eqref{eq_FokkerPlanck}, where we use an adapted strategy based on Gauss-Radau points.

\subsection{Novelty and relation to previous work}
 
While the basic building blocks of the scheme considered here are known, they are combined in a particular manner suitable for low-rank approximations, and the main novelty lies in a careful choice of error tolerances and rank truncation thresholds so that guaranteed convergence is achieved while maintaining estimates in terms of the ranks of best low-rank approximations of comparable accuracy. 

Such bounds have been obtained previously for a specific type of method based on a space-time variational formulation. In such methods, instead of proceeding by time steps to update a spatial approximation, a joint approximation in spatial and temporal variables is gradually refined towards a sufficiently accurate approximation of the entire evolution. In the context of low-rank approximations, such methods were considered in \cite{Andreev:12,BEEN19,DEG24}, though without estimates for the computed ranks. However, such approaches are very well suited to ensure an appropriate balance between approximation errors and tensor ranks, and a method ensuring near-optimal ranks for high-dimensional parabolic problems has been proposed and analyzed in \cite{BF24}.
However, suitable well-posed space-time formulations do not exist for every problem of interest, and they often strongly constrain the norms in which approximation errors can be controlled.

Here we are thus interested in methods that do not require a well-posed space-time variational formulation as a starting point.
With this restriction, to the best of our knowledge, all previously known methods either ensure convergence or control ranks, but not both.
 
Dynamical low-rank approximation \cite{KL07} falls into the latter category of controlling ranks but not convergence.
It is based on approximations by substitute problems on manifolds of matrices (or tensors) of fixed ranks. 
In such schemes, the right-hand side is projected onto the tangent space of the current base point on the manifold, and error bounds can only be obtained under the assumption that the error in this projection is small, which need not be the case (see, for example, \cite[Rem.~7.2]{Bachmayr23} and \cite[Ex.~1]{LCK25}).
The Basis Update \& Galerkin (BUG) method \cite{CL22} (see also \cite{NR25} for high order extensions) is closely related to splitting integrators for dynamical low-rank approximation \cite{LubichOseledets:14}, but can be naturally modified to include the adaptation of ranks \cite{CKL22}. However, the issue that errors of tangent space projections need to be assumed to be small is still present in these methods.

There also exists a variety of methods that directly use low-rank representations of the right-hand side, rather than only its projections.  Step-truncation methods are based on performing time steps of a standard scheme (such as a Runge-Kutta method) in tensor format and then performing rank truncation, as in \cite{DRV21,RV:23,LCK25}. A variation with conceptual similarities to BUG integrators are methods that use the right-hand side for basis enrichment and thus avoid the restrictions imposed by tangent space projections, as in \cite{YW20,AC25}.
While in principle, such methods can be ensured to converge, the rank truncation tolerances need to be tightened with decreasing time step length, and the resulting ranks are unclear.

An important conceptual ingredient for the feasibility of our approach for PDE problems are concepts related to exponential integrators \cite{KT05}, in particular in a form originally suggested by \cite{Lawson67}.
Modern exponential integrators, however, generally rely on inverses of operators in the approximation of exponentials. Such methods have been used, for example, in combination with approximations on fixed-rank manifolds in \cite{CV:23,SV:24}. 
Inverses of operators such as the Laplacian, however, are generally not easy to realize in low-rank format -- in fact, one of the most reliable techniques for obtaining low-rank approximations of such inverses consists in a reduction to sums of exponentials, which in view of \eqref{eq:expfactors} is easier to handle in low-rank format. For this reason, unlike standard exponential integrators, the methods considered here rely only on operator exponentials of Kronecker rank one.

While we prove our main results for a scheme based on Picard iteration, we also consider a variant based on spectral deferred correction (SDC), an approach that was originally proposed in \cite{DGR00}, see also \cite{LM05,HZ07,Buvoli20}. Connections to Runge-Kutta methods and Picard iteration are highlighted in particular in \cite{Weiser15,CS19}. SDC has also recently been used in the context of low-rank methods for the construction of higher-order step-truncation methods \cite{LJC24}, considering also soft thresholding following \cite{BS17}, but without an analysis of approximation ranks.

The approach can be interpreted as a hybrid between time-stepping schemes and methods based on space-time variational formulations. 
It retains flexibility in its application to different classes of PDEs and achieves error control pointwise in time.

\subsection{Outline} 

In Section \ref{sec:fixedpoint}, we begin by devising the integral formulations and their discretizations on subintervals that form the basis of our schemes.
In Section \ref{sec:thresholding}, we describe rank reduction techniques that we use and their basic properties.
The proof of our main result is conducted in Section \ref{sec:analysis}.
In Section \ref{sec:parabolic}, we consider the adaptation of our approach to parabolic problems.
Finally, we present results of numerical experiments in Section \ref{sec:numexp} and conclude with some open problems.

\section{Iterative fixed-point methods}\label{sec:fixedpoint}

Our aim is to construct high-order methods for low-rank representations of evolution problems 
\begin{equation*}
	 \partial_t u(t) = \mathcal{A}_t \bigl( u(t) \bigr)
\end{equation*}
on $L^2(\Omega)$, with a potentially unbounded operator $\mathcal{A}_t$.
In particular, we focus on the case of affine linear differential equations of the form
\begin{equation*}
  \mathcal{A}_t \bigl( v(t) \bigr) = ( \mathcal L + \mathcal G_t ) v(t) + f(t)
\end{equation*}
 where $\mathcal L$ and $\mathcal G_t$ each denote linear operators on $L^2(\Omega)$ and $f(t)$ might encode some additional nonlinear information. We further assume that the operator $\mathcal L$ captures the majority of the stiffness of the resulting differential operator. In order to simplify the notation, we denote by $\| \cdot \|$ the $L^2(\Omega)$ norm and by $\langle \cdot, \cdot\rangle$ the corresponding inner product. Additionally, we define
 \begin{equation*}
 	\mathcal W(s,t) = L^\infty \bigl(s,t; L^2(\Omega)\bigr), \quad \text{for} \quad 0 \leq s \leq t \leq T.
 \end{equation*}
 
 While it is rather complicated to provide both rank bounds and a sufficient accuracy, for general time-stepping methods, works like \cite{BS17} and \cite{bachmayr_adaptive_2015} suggest that iterative solvers such as fixed-point iterations in low-rank format are able to retain quasi-optimal rank control on the iterates. The idea would thus be to subdivide the time interval $[0,T]$, on which we would like to evolve the above problem, into $n \in \N$ smaller intervals of size $h > 0$ and to approximate solutions on those iteratively using a fixed-point iteration combined with low-rank techniques. Once a sufficiently accurate solution has been obtained on a subinterval, one can move to the next one and repeat the procedure.
 
 A classical approach for bounded $\mathcal A_t$ would consist in using the integral formulation
 \begin{equation}
 	u(t) = u(t_0) + \int_{t_0}^{t} \mathcal A_t(u(s)) \, ds \label{picard_iter}
 \end{equation}
for $t_0 = 0$ and Picard iteration to successively obtain approximations at the discrete points $h, 2h, \dots, Nh \in [0,T]$. This basic scheme is then often combined with approximating the exact solution $u(t)$ on a resulting subinterval of the form $t \in [nh, (n+1)h]$ using Chebyshev polynomial series, see for example \cite{clenshaw_numerical_1957} and \cite{clenshaw_solution_1963} for early contributions.
 
 It is well-known that stability constraints for the convergence of fixed-point iterations depend strongly on the Lipschitz constant of the operator $\mathcal{A}_t$. It is thus natural to try to reformulate the integral formulation, which is used as the base of iterative solvers in order to reduce the imposed step size restriction, since the ranks of computed low-rank approximations can potentially increase exponentially in the number of time intervals.
 
\subsection{Modified fixed-point formulation}

In physical applications, such as the time-de\-pen\-dent Schrödinger equation, it turns out that the stiffness of the time-independent linear operator $\mathcal L :
L^2(\Omega) \rightarrow L^2(\Omega)$ is mainly responsible for the magnitude of the Lipschitz constant. Applying the discrete variation of constants formula to \eqref{picard_iter}, one obtains
\begin{equation*}
	u(t) = e^{(t-t_0) \mathcal L}u(t_0) + \int_{t_0}^{t} e^{(t - s) \mathcal L} \bigl(\mathcal G_s u(s) + f(s) \bigr) \, ds.
\end{equation*}
The favorability of this formulation can be seen by comparing the resulting Lipschitz constants of our main application.

\begin{example}[Schrödinger equation with time-dependent potential]
	In the setting of the Schrödinger equation \eqref{eq_Schrodinger} with a potentially time-dependent potential $V_t$ we have 
	\begin{equation*}
		\mathcal L = i \Delta, \quad \mathcal G_t = -i V_t \quad \text{and} \quad f = 0.
	\end{equation*}
	The resulting Lipschitz constant 
	\begin{equation*}
		\norm{ -\Delta + V_t }_{\mathcal W(0,T) \rightarrow \mathcal W(0,T)}
	\end{equation*}
	is unbounded due to the Laplacian and in computations, it will depend strongly on the spatial discretization: if we use for example an $L^2$-orthonormal basis expansion with sine product basis functions, the Laplacian in matrix form becomes a diagonal operator with entries being proportional to $\Delta_{k,k} \eqsim k^2$, which is mainly responsible for the resulting stiffness. Thus, we would need to progressively decrease the stepsize $h$ in order to obtain a better spatial resolution.
	
	This issue can be circumvented by instead considering \textit{Duhamel's formula}
	\begin{equation*}
		u(t) = e^{i(t - t_0) \Delta}u(t_0) - i \int_{t_0}^{t} e^{i(t - s)\Delta} V_s u(s) \, ds,
	\end{equation*}
	which is equivalent to the Picard integral formula
	\begin{equation}
		v(t)  = v(t_0) + \int_{t_0}^t \mathcal F_s v(s) \, ds,\qquad \mathcal F_s=-ie^{-i s \Delta} V_s e^{i s \Delta},
		\label{twisted_picard}
	\end{equation}
	for the \textit{twisted variable} $v(t) = e^{-it \Delta}u(t)$ which formally satisfies
	\begin{equation}
	\label{eqmathcalF}
		\partial_t v(t) = \mathcal F_t v(t)=-ie^{-it \Delta} V_t e^{it \Delta} v(t).
	\end{equation}
	This change of variable, also known as Lawson method for Runge-Kutta methods \cite{Lawson67}, reduces the stiffness of the PDE by explicitly integrating the term in $-i\Delta$. Here, $V_t$ should be understood as an operator and not as a function, hence $e^{-it\Delta}V_t g (t) = e^{-it\Delta}(V_t g)(t)$ for $g \in \mathcal W(0,T)$. In this setting, we apply the resulting scheme on the twisted variable and after having obtained an approximation on the whole time interval $[0,T]$, we can revert the applied exponential transformations and `untwist' the approximations.
	
	Due to the norm preservation property 
	\begin{equation*}
		\| e^{i t \Delta}u(t) \| = \| u(t) \|, \quad \text{for all} \quad t\in \mathbb{R},
	\end{equation*}
	the Lipschitz constant reduces to $(t_0 - t)$ multiplied with
	\begin{equation}
		 C_V =  \| V_t \|_{\mathcal W(0,T) \rightarrow \mathcal W(0,T)}. \label{C_V}
	\end{equation}
	In particular, it is robust with respect to the spatial discretization for a bounded potential. 
Note that more realistic potentials with Coulomb singularities can be treated using Strichartz estimates if we consider $L^2(\Omega)$ norms in time, see \cite{DEG24}.
In the following, we restrict our considerations to the case of a time-independent potential for notational convenience, even though the framework can also be extended to the case of a time-dependent potential. 
\end{example}

\begin{example}[Second-order parabolic equation with anisotropic diffusion] \label{ex:mod_pf_parabolic}
	We also consider a parabolic working example on $L^2(\Omega_1 \times \Omega_2)$ with anisotropic diffusion, given by
	\begin{equation}
		\partial_t u(t) = {\rm div}\biggl( \begin{pmatrix}a&b\\b&a\end{pmatrix} \nabla u(t)\biggr)
		+ f \label{eq_parabolic}
	\end{equation}
	for constants $a>2b>0$ and a source term $f \in L^2(\Omega_1 \times \Omega_2)$, which
	can be decomposed as
	\begin{equation*}
		\mathcal L = a \Delta, \quad \mathcal G_t = \mathcal G = 2b \partial_{x_1} \partial_{x_2} \quad \text{and} \quad f(t)=f.
	\end{equation*}
	We can again use Duhamel's formula, here in the form
	\begin{equation}
	\label{parabolic_duhamel}
	u(t) = e^{a(t-t_0) \Delta }u(t_0) + \int_{t_0}^{t} e^{a (t - s) \Delta } \bigl( 2 b \partial_{x_1} \partial_{x_2}  u(s) + f \bigr) \, ds,
	\end{equation}
	as our starting point. However, an approach using twisted variables as in \eqref{twisted_picard} is not applicable due to the unboundedness of $e^{- t \Delta}$ for $t>0$, and for the same reason, additional care is needed to use $e^{a (t - s) \Delta }$ only when $t>s$.
\end{example}

Since parabolic problems thus require a separate treatment, we postpone the discussion of this case to Section \ref{sec:parabolic}. 
In the following, we first focus on the time-dependent Schrödinger equation.

\subsection{Fixed-point formulations}

Given a subinterval of the form $[t_0,t_0+h]$, a (complex-valued) function $u\in \mathcal W(t_0,t_0+h)$, and a strictly increasing sequence of nodes
\[
t_0\leq t_1<\dots <t_J\leq t_0+h
\]
for some $J \in \mathbb N$, we can define the Lagrange polynomial interpolation of $u$ as
\[
t\mapsto \sum_{j=1}^{J} u(t_j) \ell_j(t),\qquad \ell_j(t)=\prod_{m\neq j}\frac{t-t_m}{t_j-t_m}.
\]

\subsubsection{Gauss-Legendre-Picard iteration}

A classical fixed-point iteration following directly from the Picard integral formulation in the twisted variables \eqref{twisted_picard} is then given by
\begin{equation}
	\Phi_{\mathrm{Picard}} (\mathbf v)(t) = v(t_0) + \int_{t_0}^{t}  \sum_{j=1}^{J} \mathcal F_{t_j}v_j \ell_j(s)\, ds, \label{eq:extension}
\end{equation}
for $\mathbf v=(v_j)_{1\leq j\leq J}$ an approximation of $(v(t_j))_{1\leq j\leq J}$.
This allows us to formulate a fixed-point application $\mathbf \Phi_{\mathrm{Picard}}:L^2(\Omega)^J\to L^2(\Omega)^J$ as
\begin{equation} \label{twisted_fixed_point_iter}
\mathbf \Phi_{\mathrm{Picard}}(\mathbf v)_j=\Phi_{\mathrm{Picard}} (\mathbf v)(t_j)=v(t_0)+ \sum_{m = 1}^{J} \omega_{j,m} \mathcal F_{t_m} v_m,
\end{equation}
where the quadrature weights are given by
\begin{equation*}
	\omega_{j,m} = \int_{t_0}^{t_j} \ell_m(s) \, ds, \quad \text{for} \quad 1 \leq j, m \leq J.
\end{equation*}

Naturally, one would like to maximize the degree of accuracy of the quadrature formula
\begin{equation}
\label{quadrature_approx}
	\int_{t_0}^{t_0+h} g(s) \, ds \approx  \sum_{m=1}^J \omega_m g(t_m), \quad \text{for} \quad g \in \mathcal W(t_0, t_0+h),
\end{equation}
when choosing the quadrature nodes $t_1, t_2, \dots, t_J$ and corresponding weights
\[
\omega_m = \int_{t_0}^{t_0 + h} \ell_m (s) \, ds, \quad \text{for} \quad 1\leq m\leq J.
\]
It is well-known that the maximal degree of accuracy, which can be attained by such a quadrature formula is $2J-1$, and that this bound can be reached only by the so-called \textit{Gaussian quadrature formulas}, see for example \cite{stoer_introduction_1980} for scalar-valued functions. 

\begin{example}[Gauss-Legendre quadrature]
	For this family of quadrature formulas, the $J$ nodes are defined as the roots of the $J$th Legendre polynomial, which are contained in the open interval $(t_0, t_0+h)$.
	
	The Gauss-Legendre quadrature nodes and weights can be computed efficiently using the Golub-Welsch algorithm \cite{GW69}. Defining the tridiagonal matrix
	\begin{equation}
		T_J = \begin{pmatrix}
			t_0+h/2 & \beta_1 & & & & 0 \\
			\beta_1 & t_0+h/2 & \beta_2 & & &  \\
			& \beta_2 & t_0+h/2 & \ddots & &   \\
			& & \ddots & \ddots & \beta_{J-2}&  \\
			& & & \beta_{J-2} & t_0+h/2 & \beta_{J-1} \\
			0 & &  & & \beta_{J-1} & t_0+h/2
		\end{pmatrix}, \label{tridiaggauss}
	\end{equation}
	where $\beta_m = \frac{mh}{2 \sqrt{4m^2-1}}$, it can be shown that the $J$ eigenvalues of $T_J$ correspond exactly to the Gauss-Legendre nodes $t_1, t_2, \dots, t_J$ and that the weights can be expressed as
$\omega_m = 2v_{m,1}^2 / \| v_m \|^2_2,$
	where $v_{1,1}, v_{2,1}, \dots, v_{J,1}$ denote the first entries of the eigenvectors $v_1, v_2, \dots, v_J$, respectively. 
\end{example}

In the following, we thus restrict our considerations to the use of the Gauss-Legendre quadrature for the approximation of the time-dependent Schrödinger equation. The advantages of these particular nodes will be discussed in the subsequent sections.

\begin{remark}[Spectral integration matrix]
We refer to \cite{DGR00} and the references therein for a discussion on the numerical stability properties of the \textit{spectral integration matrix} $(\omega_{j,m})_{1\leq j,m\leq J}$. By translation and scaling symmetries, its entries do not depend on $t_0$ and are proportional to $h$; they can be precomputed exactly using Gauss-Legendre quadrature in the smaller intervals $[t_0,t_j]$, $1 \leq j \leq J$.
\end{remark}

\begin{remark}[Extended fixed-point iteration and transition to boundary value] \label{rm:extended_fp}
To transition to the next interval,
following the approach described in \cite{HZ07}, we
approximate the boundary value by
\begin{equation}
\Phi_{\mathrm{Picard}} (\mathbf v)(t_0+h) =
v(t_0) + \sum_{m=1}^{J} \omega_m  \mathcal{F}_{t_m} v_m. \label{eq:boundary}
\end{equation}
One could also define the fixed-point formulation on the $J+1$ nodes $t_1,\dots,t_J,t_0+h$, 
however the last value is not actively used inside the fixed-point iterations.
\end{remark}

\subsubsection{Spectral Deferred Correction}

Another possible fixed-point iteration, first introduced in \cite{DGR00}, is spectral deferred correction.
Observe that \eqref{twisted_fixed_point_iter} can be rewritten inductively in $j$ as
\[
\mathbf \Phi_{\mathrm{Picard}}(\mathbf v)_j=\mathbf \Phi_{\mathrm{Picard}}(\mathbf v)_{j-1}+\sum_{m=1}^J \tilde \omega_{j,m}\mathcal F_{t_m}v_m,
\]
where $\mathbf \Phi_{\mathrm{Picard}}(\mathbf v)_0=v_0$ and
\begin{equation*}
	\tilde{\omega}_{j,m} = \omega_{j,m} - \omega_{j-1,m} = \int_{t_{j-1}}^{t_j} \ell_m(s) \, ds.
\end{equation*}
Given a low-order integrator $\Psi$ such that $\Psi_{s,t}(u(s))\approx u(t)$, SDC is constructed by adding to the above formula a correction, obtained by integrating the residual error from $t_{j-1}$ to $t_j$, yielding an approximation of higher order  \cite{Buvoli20}.
In order to keep control on the ranks of the updates,
we restrict our considerations
to explicit low-order integrators $\Psi$. The SDC iterations can then be written as
\begin{equation}
\label{eq:SDC}
\mathbf \Phi_{\mathrm{SDC}}(\mathbf v)_j=\mathbf \Phi_{\mathrm{SDC}}(\mathbf v)_{j-1}+\sum_{m=1}^J \tilde \omega_{j,m}\mathcal F_{t_m}v_m
+\Psi_{t_{j-1},t_j}\big(\mathbf \Phi_{\mathrm{SDC}}(\mathbf v)_{j-1}-v_{j-1}\big),
\end{equation}
where $\mathbf \Phi_{\mathrm{Picard}}(\mathbf v)_0=v_0$.
While the explicit Euler method
\begin{equation*}
	\Psi_{t_{j-1},t_j} = (t_j-t_{j-1})\mathcal F_{t_{j-1}}
\end{equation*}
seems quite standard \cite{CS19},
in the case of the Schrödinger equation, or more generally Hamiltonian systems, we prefer the symmetric integrator
\begin{equation*}
\Psi_{t_{j-1},t_j} = (t_j-t_{j-1})\mathcal F_{(t_{j-1}+t_j)/2}
\end{equation*}
motivated by \cite{KD25}, since the midpoint approximation of $\mathcal F$ is expected to be more accurate. 
Note that we do not exactly preserve the energy, time-reversibility or symplecticity, since the integrator remains explicit.
However, we only need to apply the potential operator $V$ once, which is advantageous for the rank control.

The transition to the subinterval boundary is then carried out analogously to \eqref{eq:boundary}.
 
\subsection{Reformulation in matrix form}

In order to better compare the Picard and SDC iterations and understand how the corresponding fixed-point formulations are defined on the Gauss-Legendre nodes, following \cite{huang_accelerating_2006}, we derive matrix versions of the corresponding fixed-point formulations.

For the Picard iteration, we simply have
\begin{equation}
\mathbf \Phi_{\mathrm{Picard}} \mathbf v= \mathds{1} \otimes v_0 + \mathbf{\Sigma} \mathbf v,
\label{eq:matrix_picard}
\end{equation}
where $\mathbf{\Sigma}=(\omega_{j,m} \mathcal F_{t_m})_{1\leq j,m\leq J}$.
Regarding SDC, we start by iterating \eqref{eq:SDC} to obtain
\begin{equation*}
\mathbf \Phi_{\mathrm{SDC}}(\mathbf v)_j = \mathbf \Phi_{\mathrm{Picard}}(\mathbf v)_j + \sum_{i=1}^j  \Psi_{t_{i-1}, t_i} \bigl( \mathbf \Phi_{\mathrm{SDC}}(\mathbf v)_{i-1} - v_{i-1} \bigr).
\end{equation*}
One can then rearrange the above equation as
\begin{equation}
	\bigl( \mathbf{I} - \mathbf{\Psi} \bigr) \mathbf \Phi_{\mathrm{SDC}}(\mathbf v) = \mathds{1} \otimes v_0 + \bigl( \mathbf{\Sigma} - \mathbf{\Psi} \bigr) \mathbf{v},\label{eq:matrix_sdc}
\end{equation}
where the matrix $\mathbf\Psi=(\mathds 1_{j>i}\Psi_{t_{i}, t_{i+1}})_{1\leq j,i\leq J}$
is lower triangular, allowing to explicitly compute $\mathbf \Phi_{\mathrm{SDC}}(\mathbf v)$ one coordinate at a time.

These matrix equations will allow us to adapt the soft thresholded fixed point method from \cite{BS17}, originally designed for the Richardson iteration, to our setting of time integration, as will be seen in particular in Lemma~\ref{lem_rank_uk}.

\subsection{Collocation method limit of the fixed-point iterations}

Considering the formulations \eqref{eq:matrix_picard} and \eqref{eq:matrix_sdc}, it becomes evident that the solutions to the fixed-point problems $\mathbf \Phi_{\mathrm{Picard}}(\mathbf v^*)=\mathbf v^*$ and $\mathbf \Phi_{\mathrm{SDC}}(\mathbf v^*)=\mathbf v^*$ coincide and satisfy
\begin{equation}
\label{eq_fixed_point}
	\bigl( \mathbf{I} - \mathbf{\Sigma} \bigr) \mathbf{v^*} = \mathds{1} \otimes v_0.
\end{equation}
It turns out that this equation is equivalent to an implicit collocation formulation, see also \cite{huang_accelerating_2006}.
Indeed, it holds that $\mathbf v^*=\big(\Phi_{\mathrm{Picard}} (\mathbf v^*)(t_j)\big)_{1\leq j\leq J}$, where $\Phi_{\mathrm{Picard}} (\mathbf v^*)$ is a polynomial of degree $J$ in time satisfying
\begin{equation}
\label{collocation}
\partial_t \Phi_{\mathrm{Picard}} (\mathbf v^*)(t_m)= \sum_{j=1}^{J} \mathcal F_{t_j}v_j^* \ell_j(t_m)=\mathcal F_{t_m}v_m^*=\mathcal F_{t_m}\big(\Phi_{\mathrm{Picard}} (\mathbf v^*)(t_m)\big),
\end{equation}
that is, $\Phi_{\mathrm{Picard}} (\mathbf v^*)$ is a solution to \eqref{eqmathcalF} at times $t_m$, for all $1\leq m \leq J$. As a consequence, it is also the solution to an implicit Runge-Kutta method with $J$ stages, whose Butcher tableau is given by
 
\begin{equation}
\label{butcher}
	\renewcommand\arraystretch{1.2}
	\begin{array}{c|c}
	\mathbf{c} & \mathbf{A}\\
	\hline
	& \mathbf{b}^T
\end{array}
\qquad\text{where}\quad
\mathbf{A}_{m,j}
=\frac{\omega_{j,m}}{h},  \quad \mathbf{b}_j
= \frac{\omega_j}{h}  \quad \text{and} \quad c_m=\frac{t_m-t_0}{h},
\end{equation}
see \cite[Theorem 7.7]{hairer_solving_1993}.
Such methods are of particular interest, since they can reach, depending on the underlying quadrature, up to double the order of convergence of the number of interpolating conditions or stages, that is, $2J$. The upper bound can only be reached using Gaussian quadrature formulas, in which case the resulting Runge-Kutta methods are denoted as \textit{Gauss methods}.

As for any Runge-Kutta scheme, we can analyse the linear stability by considering the scalar test problem
\begin{equation}\label{eq:exptestproblem}
\dot y(t)= \lambda y(t), \qquad \lambda \in \mathbb C,
\end{equation}
 for which the appoximation generated by the scheme after one time step $h$ is given by $R(h\lambda)y(0)$ for some rational function $R$.
 It can be shown that the Gauss collocation method defined by \eqref{butcher} is
\textit{A-stable}, that is, the stability function $R$ satisfies
 \begin{equation*}
 	|R(z)| \leq 1 \quad \text{for all} \quad z \in \mathbb C \quad \text{with} \quad \text{Re} (z) \leq 0.
 \end{equation*}
It is in addition
\textit{isometry-preserving} in the sense that
 \begin{equation*}
 	|R(z)| = 1 \quad \text{for} \quad \text{Re}(z) = 0.
 \end{equation*}
As the Schrödinger equation preserves the $L^2(\Omega)$ norm, the eigenvalues of its evolution operator are purely imaginary, and the above property translates into preservation of norm for the Gauss method. This will be addressed further in Lemma~\ref{lm:isometry}.
Altogether, these properties make the Gauss-Legendre nodes an adequate choice in our setting.

 \begin{remark}[Original formulation of SDC]
	In the original paper \cite[Rem. 4.1]{DGR00} the transition to the subinterval boundary is carried out by evaluation of the interpolation polynomial, that is, $\sum_{m=1}^J v^k_m\ell_m(t_0+h)$. However, as  observed in several numerical experiments, the resulting method is not isometry-preserving, which is particularly important when computing approximations to the time-dependent Schrödinger equation. Moreover, the resulting scheme is then also no longer equivalent to a Gauss method and loses the superconvergence property at the subinterval boundaries.
\end{remark}

 \begin{remark}[Advantages of using SDC over Picard iteration]
	It was shown in \cite{CS19} that both Picard iteration and SDC (using the explicit Euler as the low-order integrator) pick up at least one order of accuracy in each iteration. However, SDC is expected to converge even faster in the neighborhood of the fixed point, allowing one to reduce the number of iterations.
	Additionally, \cite{CS19} suggest that using higher-order base solvers can indeed lead to the SDC iteration being able to pick up more than one order of accuracy with each iteration. Moreover, the used solver also has a direct influence on the resulting stability properties of the scheme.
	It was also proven in \cite{HZ07} that if the fixed-point method is iterated until convergence, the maximal order of convergence $2J$ for Gauss-Legendre nodes, can in fact be reached, see also Lemma~\ref{lm:locerr2J}. In that perspective, SDC can also be viewed as a preconditioner for the Picard iteration, as can be deduced from the matrix formulation \eqref{eq:matrix_sdc}.
\end{remark}
 
 \subsection{Convergence analysis of the fixed-point iterations}
 
We are now in the position to show the contractivity of the constructed fixed-point iterations,
for small enough step sizes~$h$.
 For $\mathbf v\in L^2(\Omega)^J$, define the norm
 \[
 \|\mathbf v\|_J=\max_{1\leq j\leq J}\|v_j\|,
 \]
 where we recall that $\|\cdot\|=\|\cdot\|_{L^2(\Omega)}$.
 \begin{prop}[Contraction constant for the Schrödinger equation]
 \label{prop_contractivity}
 	Let
 	\begin{equation}
	\label{Lambda_J}
 		\Lambda_J = \frac{1}{h} \int_{0}^{h} \sum_{m=1}^J|\ell_m(s)| \, ds \quad\text{and}\quad \rho = h \Lambda_J C_V.
 	\end{equation}
 	Then the fixed-point iterations $\mathbf \Phi_{\mathrm{Picard}}$ and $\mathbf \Phi_{\mathrm{SDC}}$ defined by \eqref{twisted_fixed_point_iter} and \eqref{eq:SDC} satisfy
 	\begin{equation*}
	\|\mathbf \Phi_{\mathrm{Picard}}\mathbf v-\mathbf \Phi_{\mathrm{Picard}}\mathbf w\|_{\mathcal W(0,h)}\leq \rho \|\mathbf v-\mathbf w\|_J
	\end{equation*}
	and
	\begin{equation*}
	\|\mathbf \Phi_{\mathrm{SDC}}\mathbf v-\mathbf \Phi_{\mathrm{SDC}}\mathbf w\|_J\leq 2\rho e^\rho \|\mathbf v-\mathbf w\|_J
 	\end{equation*}
	for any $\mathbf v,\mathbf w\in L^2(\Omega)^J$.
 \end{prop}
 
 \begin{proof}
 For the Picard iteration, we see that for any time $t\in [0,h]$, the expression
\[
\| \Phi_{\mathrm{Picard}}({\bf v})(t)-\Phi_{\mathrm{Picard}}({\bf w})(t) \|
= \biggnorm{ \int_0^{t} \sum_{m=1}^J \ell_m(s) \mathcal F_{t_m}(v_m-w_m)\,ds }
\]
is bounded by
\[
\int_0^{t}  \sum_{m=1}^J |\ell_m(s)|\,ds \ C_V \|v_m-w_m\| \leq \rho \, \|{\bf v}-{\bf w}\|.
\]
 For the SDC iteration, letting $\phi_j=(\mathbf \Phi_{\mathrm{SDC}}\mathbf v)_j-(\mathbf \Phi_{\mathrm{SDC}}\mathbf w)_j$, we have
 	\begin{align*}
 		\| \phi_j \| &\leq \| \phi_{j-1} \|
 		 + \bigl\| \Psi_{t_{j-1},t_j} \bigl( \phi_{j-1} - v_{j-1} + w_{j-1} \bigr) \bigr\| 
		 + \sum_{m=1}^{J} | \tilde{\omega}_{j,m}|\, \bigl\| \mathcal{F}_{t_m} \bigl( v_m - w_m \bigr) \bigr\|\\
 		&\leq  (1+\tau_j C_V ) \bigl\| \phi_{j-1} \bigr\| + \tau_j C_V \bigl\| v_{j-1} - w_{j-1} \bigr\|
		+\sum_{m=1}^J |\tilde{\omega}_{j,m}| \, C_V \bigl\| v_m - w_m \bigr\|.
 	\end{align*}
 	By induction, using the upper bounds $1+c\leq e^c$ for all $c \in \R$ and $\sum_{j=1}^J\tau_j\leq h$, we obtain
 	\begin{equation*}
 		\max_{1\leq j\leq J} \bigl\| \phi_{j} \bigr\|
 		\leq e^{h C_V}\Bigl(\sum_{j,m=1}^J |\tilde{\omega}_{j,m}|+h \Bigr) C_V \bigl\|\mathbf v -\mathbf w\bigr\|_J,
 	\end{equation*}
 	and we conclude by observing that
 	\begin{equation*}
 		\sum_{j,m=1}^J |\tilde{\omega}_{j,m}|=\sum_{j=1}^J\int_{t_{j-1}}^{t_{j}}\sum_{m=1}^J|\ell_m (s)| \, ds \leq h\Lambda_J
		\quad\text{and}\quad
		1= \frac{1}{h} \int_{0}^{h} \sum_{m=1}^J\ell_m(s) \, ds\leq \Lambda_J.\qedhere
 	\end{equation*}
 \end{proof}

 \begin{remark}[Choice of quadrature nodes] \label{rem:avglebesgue}
 	The constant $\Lambda_J$ defined in \eqref{Lambda_J} can be bounded from above by the Lebesgue constant. The asymptotic behavior of the  latter is in the order $O (\sqrt{J})$ in the case of the Gauss-Legendre and Radau-Legendre quadrature points, see \cite{hager_lebesgue_2016}. Since this constant is a limiting factor for the step size $h$, one could also think of using Chebyshev nodes, whose asymptotic behavior is in $\mathcal O (\log J)$. Unfortunately, using Chebyshev nodes does not result in a scheme of order $2J$ at the subinterval boundaries. However, the averaged constant $\Lambda_J$ differs much less between the Gauss-Legendre and Chebyshev points, and can in both cases be bounded by 2 for $J \leq 10$, which is sufficient for our purposes.
 	
 	Moreover, the numerical experiments performed in \cite{aristoff_orbit_2014} for the particular application of celestial mechanics suggest that the usage of an iteration based on Gauss-Legendre nodes leads to improved convergence properties over the Chebyshev nodes.
 \end{remark}

\section{Thresholding strategy and resulting algorithms}\label{sec:thresholding}

We now turn to our strategy for reducing the ranks of the intermediate approximations. Since we use a fixed-point iteration in order to successively approximate the true solution, it seems natural to combine the iterations with a rank control strategy. 

\subsection{Rank reduction}

For every $v \in L^2(\Omega_1 \times \Omega_2)$, there exists a singular value decomposition 
\begin{equation}\label{eq:svd}
  v = \sum_{k=1}^\infty \sigma_k v^{(1)}_k \otimes v^{(2)}_k,
\end{equation}
with the singular values $\sigma_1 \geq \sigma_2 \geq \ldots\geq 0$ and orthonormal families of singular vectors $ ( v^{(1)}_k )_{k \in \N}$ in $L^2(\Omega_1)$ and $( v^{(2)}_k)_{k \in \N}$ in $L^2(\Omega_2)$. We have already noted the well-known fact that the truncation of the singular value decomposition to rank $r$ yields the best approximation of rank $r$ in $L^2(\Omega_1\times \Omega_2)$, with
\[
  \biggnorm{  v -  \sum_{k=1}^r \sigma_k v^{(1)}_k \otimes v^{(2)}_k }^2 = \sum_{k > r} \sigma_k^2.
\]

We define the nonlinear \textit{hard thresholding} operator on $L^2(\Omega_1 \times \Omega_2)$ by
\begin{equation}\label{eq:hardthresh}
	\mathcal{H}_\alpha(v) =  \sum_{\substack{k \in \N \\ \sigma_k > \alpha}}  \sigma_k\, v^{(1)}_k \otimes v^{(2)}_k,
\end{equation}
and the corresponding \textit{recompression} operator by
\begin{equation} \label{eq:recompression}
	\mathcal{R}_\delta(v) =  \sum_{1 \leq k \leq r_{\delta}}  \sigma_k\, v^{(1)}_k \otimes v^{(2)}_k,
	\quad \text{where} \quad r_\delta = \min \Bigl\{ r \in \N_0 \colon  \sum_{k > r} \sigma_k^2 \leq \delta^2 \Bigr\}.
\end{equation}

In general, the addition of two low-rank tensor representations leads to an increase in the rank parameter: the sum of two elements in $L^2(\Omega_1 \times \Omega_2)$ with respective ranks $r_1$ and $r_2$
has rank up to $r_1 + r_2$.
Even though the used low-rank format is closed under addition, the resulting representation may be redundant. This redundancy can be removed in practice by applying a recompression operator \eqref{eq:recompression} with a sufficiently small tolerance $\delta$ to the sum.

\begin{figure}[t] 
	\centering
	\begin{tikzpicture}[scale=0.7, font=\small]
		
		\draw[->] (0,0) -- (11,0) node[pos=1.05] {$i$};
		\draw[->] (0,0) -- (0,5) node[anchor=south, pos=1] {$\sigma_i$};
		
		\draw[red, dashed] (0,1) -- (11,1);
		
		\draw[red] (0,1) -- (0.2,1); 
		\node[left, red] at (0,1) {$\alpha$};
		
		\foreach \i in {1,...,10} {
		
			\pgfmathsetmacro{\val}{5*exp(-0.3*(\i-1))}
			
			\fill (\i,\val) circle (2pt);
			
			\pgfmathsetmacro{\tempDiff}{\val - 1}
			\pgfmathsetmacro{\shiftVal}{max(\tempDiff, 0)}
			\fill[red] (\i,\shiftVal) circle (2pt);
			
			\ifdim \tempDiff pt>0pt
			\draw[->, red] (\i,\val) -- (\i,\shiftVal);
			\fi
			
			\begin{scope} [xshift=220, yshift=100]
				
				\draw[white, fill=white, opacity=0.8] (0,0) rectangle (3,1.6);
				\draw (0,0) rectangle (3,1.6);

				\fill (0.3,1.2) circle (2pt);
				\node[right] at (0.5,1.2) {$\sigma_{i}(v)$};
				
				\fill[red] (0.3,0.4) circle (2pt);
				\node[right] at (0.5,0.4) {$\sigma_{i}\big(\mathcal S_\alpha (v)\big)$};
			\end{scope}
		}
	\end{tikzpicture}
	\caption{Illustration of the soft thresholding operation applied to some element $v \in L^2(\Omega_1 \times \Omega_2)$.}
	\label{fig:st}
\end{figure}
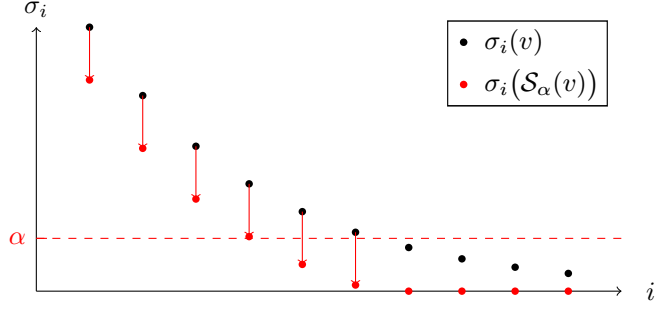

Another possibility to reduce the rank of an element $v \in L^2(\Omega_1 \times \Omega_2)$ would be to apply the \textit{soft thresholding} operator defined by
\begin{equation}\label{eq:softthresh}
    \mathcal{S}_\alpha(v) =  \sum_{k \in \N} \max\{  \sigma_k - \alpha, 0 \}\, v^{(1)}_k \otimes v^{(2)}_k.
\end{equation}
A visualization of the application of the operator to a given sequence of singular values can be found in Figure \ref{fig:st}. The errors resulting from the application of the truncation operators are related by
\begin{equation*}
	\| v - \mathcal S_\alpha(v) \|^2 = \| v - \mathcal H_\alpha(v) \|^2 + \alpha^2 \rank(\mathcal H_\alpha(v)). 
\end{equation*}
Note that $\mathcal S_\alpha (v)$ and $\mathcal H_\alpha(v)$ have the same rank $r=|\{k\in\N\colon\sigma_k>\alpha\}|$.

The following lemma allows to relate the errors caused by soft thresholding an element in $L^2(\Omega_1 \times \Omega_2)$ with different thresholds.

\begin{lemma} \label{lemma:thresholds}
	For $0 < \alpha \leq \beta$ it holds that
	\begin{equation*}
		\| v - \mathcal S_\alpha(v) \| \leq \| v - \mathcal S_\beta (v) \| \leq \frac{\beta}{\alpha} \| v - \mathcal S_\alpha(v) \|, 
	\end{equation*}
	for all $v \in L^2(\Omega_1 \times \Omega_2)$.
\end{lemma}

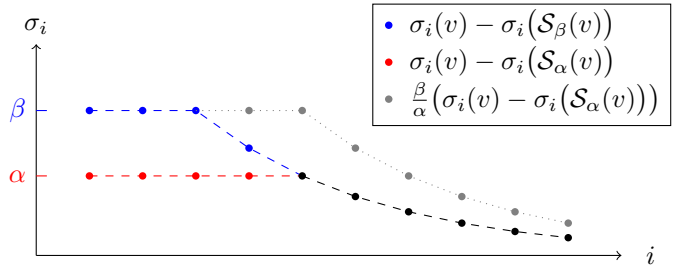
\begin{figure}[h] 
	\centering
	\begin{tikzpicture}[scale=0.7, font=\small]
		\draw[->] (0,0) -- (11,0) node [pos=1.05] {$i$};
		\draw[->] (0,0) -- (0,4) node[anchor=south] {$\sigma_i$};
		
		\pgfmathsetmacro{\valThird}{5*exp(-0.3*(3-1))}
		\pgfmathsetmacro{\valFourth}{5*exp(-0.3*(4-1))}
		\pgfmathsetmacro{\valFifth}{5*exp(-0.3*(5-1))}
		\pgfmathsetmacro{\vSix}{5*exp(-0.3*(6-1))}
		\pgfmathsetmacro{\vSeven}{5*exp(-0.3*(7-1))}
		\pgfmathsetmacro{\vEight}{5*exp(-0.3*(8-1))}
		\pgfmathsetmacro{\vNine}{5*exp(-0.3*(9-1))}
		\pgfmathsetmacro{\vTen}{5*exp(-0.3*(10-1))}
		
		\foreach \i in {1,2,3} {
			\fill[blue] (\i,\valThird) circle (2pt);
		}
		\foreach \i in {1,2,3,4} {
			\fill[red] (\i,\valFifth) circle (2pt);
		}
		
		\fill[blue] (4,\valFourth) circle (2pt);
		
		\fill[gray] (4,\valThird) circle (2pt);
		\fill[gray] (5,\valThird) circle (2pt);
		\fill[gray] (6,\vSix * \valThird / \valFifth) circle (2pt);
		\fill[gray] (7,\vSeven * \valThird / \valFifth) circle (2pt);
		\fill[gray] (8,\vEight * \valThird / \valFifth) circle (2pt);
		\fill[gray] (9,\vNine * \valThird / \valFifth) circle (2pt);
		\fill[gray] (10,\vTen * \valThird / \valFifth) circle (2pt);
		
		\draw[gray, dotted]
		(3, \valThird) -- (4, \valThird) -- (5,\valFifth * \valThird / \valFifth) -- (6,\vSix * \valThird / \valFifth) -- (7,\vSeven * \valThird / \valFifth) -- (8,\vEight * \valThird / \valFifth) -- (9,\vNine * \valThird / \valFifth) -- (10,\vTen * \valThird / \valFifth);
		
		\draw[red,dashed]
		(1,\valFifth) -- (2,\valFifth) -- (3,\valFifth) -- (4,\valFifth) -- (5,\valFifth);
		
		\draw[dashed]
		(5,\valFifth) -- (6,\vSix) -- (7,\vSeven) -- (8,\vEight) -- (9,\vNine) -- (10,\vTen);
		
		\draw[blue,dashed]
		(1,\valThird) -- (2,\valThird) -- (3,\valThird) -- (4,\valFourth) -- (5,\valFifth);
		
		\foreach \i in {5,6,7,8,9,10} {
			\pgfmathsetmacro{\val}{5*exp(-0.3*(\i-1))}
			\fill (\i,\val) circle (2pt);
		}
		
		\draw[red] (0,\valFifth) -- (0.2,\valFifth); 
		\node[left, red] at (0,\valFifth) {$\alpha$};
		
		\draw[blue] (0,\valThird) -- (0.2,\valThird); 
		\node[left, blue] at (0,\valThird) {$\beta$};
		
		\begin{scope} [xshift=180, yshift=70]
			
			\draw[white, fill=white, opacity=0.8] (0,0) rectangle (5.6,2.3);
			\draw (0,0) rectangle (5.6,2.3);

			\fill[blue] (0.3,1.9) circle (2pt);
			\node[right] at (0.5,1.9) {$\sigma_{i}(v) - \sigma_{i}\big(\mathcal S_\beta (v)\big)$};

			\fill[red] (0.3,1.2) circle (2pt);
			\node[right] at (0.5,1.2) {$\sigma_{i}(v) - \sigma_{i}\big(\mathcal S_\alpha (v)\big)$};
			
			\fill[gray] (0.3,0.5) circle (2pt);
			\node[right] at (0.5,0.5) {$\frac{\beta}{\alpha} \big(\sigma_{i}(v) - \sigma_{i}\big(\mathcal S_\alpha (v)\big)\big)$};
		\end{scope}
	
	\end{tikzpicture}
	\caption{Visualization of Lemma \ref{lemma:thresholds}: deviations between the singular values of some element $v \in L^2(\Omega_1 \times \Omega_2)$ and the output of a soft thresholding operator applied to $v$. The black dots represent coinciding values of the red and blue dots.}
\end{figure}

\begin{proof}
For $v$ of the form \eqref{eq:svd}, we have
\[
\| v - \mathcal S_\alpha(v) \|^2=\sum_{k \in \N} \min\{\alpha,\sigma_k\}^2,
\]
hence the result follows from the inequalities $\min\{\alpha,\sigma_k\} \leq \min\{\beta,\sigma_k\}\leq \frac{\beta}{\alpha}\min\{\alpha,\sigma_k\}$.
\end{proof}

\subsection{Iterative schemes}

An additional feature of the soft thresholding operator is its \textit{non-expansiveness} property, that is, for all $\alpha \geq 0$ and $v,w \in L^2(\Omega)$, it holds that
\begin{equation*}
	\| \mathcal S_\alpha(v) - \mathcal S_\alpha(w) \| \leq \| v - w \| \quad\text{and}\quad \| (\mathcal I-\mathcal S_\alpha)(v) - (\mathcal I-\mathcal S_\alpha)(w) \| \leq \| v - w \|,
\end{equation*}
where $\mathcal I$ denotes the identity operator.
We refer to \cite{BS17} and \cite{Bachmayr23} for a proof. This feature can now be exploited for the rank control within our iterative methods: using the previously introduced notation, we set for the Gauss-Legendre-Picard iteration
\begin{equation}
	\mathbf v^{k+1} = \boldsymbol{\mathcal S}_{\alpha_k} \bigl( \mathbf{\Phi}_{\mathrm{Picard}} (\mathbf v^k) \bigr) = \Big( \mathcal S_{\alpha_k} \big( \Phi \mathbf v^k(t_j) \big) \Big)_{1 \leq j \leq J}, \label{eq:picard_thresh}
\end{equation}
that is, we apply the soft thresholding operator with threshold $\alpha_k$ component-wise to each of the updated approximations at the Gauss-Legendre nodes.
Soft thresholding the fixed-point iterations as in \eqref{eq:picard_thresh}
does not interfere with the contractivity property, that is,
\begin{equation*}
	\bigl\| \mathbf v^{k+1} - \mathbf w^{k+1} \bigr\| = \bigl\| \boldsymbol{\mathcal S}_{\alpha_k} \big( \mathbf{\Phi}_{\mathrm{Picard}} (\mathbf v^k) \big) - \boldsymbol{\mathcal S}_{\alpha_k} \big( \mathbf{\Phi}_{\mathrm{Picard}} (\mathbf w^k) \big) \bigr\| \leq \rho \bigl\| \mathbf v^k - \mathbf w^k \bigr\|,
\end{equation*}
for all $\mathbf v^k, \mathbf w^k \in L^2(\Omega)^J$, where $\rho$ is contraction factor from \eqref{Lambda_J}.

For the SDC iteration, we instead apply the soft thresholding operator after updating each of the Gauss-Legendre nodes. Then \eqref{eq:SDC} modifies according to
\begin{equation}
\label{eq_SDC_phi}
	v^{k+1}_j = \mathcal S_{\alpha_k} (\phi_j^{k}),
	\qquad \phi_j^{k} = v^{k+1}_{j-1} 
	+ \sum_{m=1}^{J} \tilde{\omega}_{j,m} \mathcal{F}_{t_m} v^k_m
	+ \Psi_{t_{j-1}, t_j} \big( v^{k+1}_{j-1} - v^k_{j-1} \big).
\end{equation}
We observe in the numerical experiments that
applying $\mathcal S_{\alpha_k}$ after each of the $J$ updates
results in much smaller ranks of the intermediate approximations
than thresholding only at the end of the iteration, as we would do with $\boldsymbol{\mathcal S}_{\alpha_k} \big( \mathbf{\Phi}_{\mathrm{SDC}} (\mathbf w^k) \big)$.
However this makes the analysis more intricate, which is also the reason why we will focus on the Picard iteration in the following.

\begin{figure}[t]
	\centering
	\begin{minipage}{0.5\textwidth}
		\centering
		\begin{tikzpicture}[scale=1]
			
			\node[label=left:{$v$}] at (-5,0) {$\bullet$};
			\node[label=below:{$\Phi (v)$}] at (-3,-1) {$\bullet$};
			\node[red,label=above:{$\mathcal S_\alpha \big(\Phi (v)\big)\;\;\;\phantom.$}] at (-2,0) {$\bullet$};
			\node[label=right:{$0$}] at (0,2) {$\bullet$};
			\node[red,label=right:{$v_\alpha=\mathcal S_\alpha \big( \Phi (v_\alpha) \big)$}] at (0,0) {$\bullet$};
			\node[label=right:{$\Phi (v_\alpha)$}] at (0,-1) {$\bullet$};
			\node[label=right:{$v^*$}] at (0,-2.5) {$\bullet$};
			
			\draw[thick,->] (-5,0) 
			to[bend left=15] (-3.06,-0.97);  
			\draw[red, thick,->] (-3,-1) 
			to[bend right=15] (-2.04,-0.04); 
			\draw[red, thick,<-] (0,-0.05) 
			to[bend left=20] (0,-0.94);      
			\draw[thick,->] (0,-0.05) 
			to[bend right=20] (0,-0.94);      
			
			\draw[dotted] (-3,-1) 
			to[bend left=0] (0,-2.5);
			\draw[dotted] (-5,0) 
			to[bend right=0] (0,0);
			\draw[dotted] (-2,0) 
			to[bend left=0] (0,2);
			\draw[dotted] (-3,-1) 
			to[bend left=0] (-2,0);
			\draw[dotted] (-3,-1) 
			to[bend left=0] (-5,0);
			
			\draw[dotted] (0,0) 
			to[bend left=0] (0,2);
			\draw[dotted] (0,-2.5) 
			to[bend left=0] (0,2);
			\draw[dotted] (0,-1) 
			to[bend right=0] (0,-2.5);
			
		\end{tikzpicture}
	\end{minipage}%
	\begin{minipage}{0.5\textwidth}
		\centering
		\begin{tikzpicture}[scale=1]
			
			\node[blue, label=right:{$v_{\alpha_0}$}] at (0,0) {$\bullet$};
			\node[blue, label=right:{$v_{\alpha_1}=v_{\alpha_2}$}] at (0,-2) {$\bullet$};
			\node[blue, label=right:{$v_{\alpha_3}=v_{\alpha_4}$}] at (0,-3) {$\bullet$};
			\node[blue] at (0,-3.5) {$\bullet$};
			\node at (0.4, -3.4) {$\vdots$};
			\node[blue] at (0,-3.75) {$\bullet$};
			\node[blue] at (0,-3.875) {$\bullet$};
			\node[blue] at (0,-3.9375) {$\bullet$};
			\node[label=right:{$v^*$}] at (0,-4) {$\bullet$};
			\node[label=left:{$v^0$}] at (-4.5,0) {$\bullet$};
			\node[label=above:{$v^1=\mathcal S_{\alpha_0} \big( \Phi (v^0) \big)$}] at (-3,0) {$\bullet$};
			\node[label=left:{$v^2$}] at (-2,-2/3) {$\bullet$};
			\node[label=left:{$v^3$}] at (-4/3,-10/9) {$\bullet$};
			\node[label=left:{$v^4$}] at (-8/9,-3+34/27) {$\bullet$};
			\node[label=left:{$v^5$}] at (-16/27,-3+68/81) {$\bullet$};
			\node at (-32/81,-2.62) {$\bullet$};

			\draw[thick,->] (-4.5,0) to[bend right=15] (-3.05,0);
			\draw[thick,->] (-3,0) to[bend left=15] (-2.04,-2/3+0.02);
			\draw[thick,->] (-2,-2/3) to[bend left=15] (-4/3-0.04,-10/9+0.02);
			\draw[thick,->] (-4/3,-10/9) to[bend right=15] (-8/9-0.02,-3+34/27+0.04);
			\draw[thick,->] (-8/9,-3+34/27) to[bend right=15] (-16/27-0.02,-3+68/81+0.04);
			\draw[thick,->] (-16/27,-3+68/81) to[bend left=15] (-32/81-0.015,-2.62+0.05);
			
			\draw[thick,dashed,blue,->](0,0) to[bend left=15] (0,-2);
			\draw[thick,dashed,blue,->](0,-2) to[bend left=15] (0,-3);
			\draw[thick,dashed,blue,->](0,-3) to[bend left=15] (0,-3.5);
			
			\draw[dotted] (-4.5,0) to[bend left=0] (0,0);
			\draw[dotted] (-3,0) to[bend left=0] (0,-2);
			\draw[dotted] (-4/3,-10/9) to[bend left=0] (0,-3);
			
			\draw[dotted] (-16/27,-3+68/81) to[bend right=0] (0,-3.5);
		\end{tikzpicture}
	\end{minipage}
	\caption{Left: illustration of the fixed-point $v^*$ of $\Phi$, and of the modified version $v_\alpha$ due to the soft thresholding operator $\mathcal S_\alpha$. Right: convergence of the fixed-point iteration coupled with soft thresholding, $v^k$, when the parameter $\alpha$ is decreased every second iteration.}
	\label{fig:st_illustrations}
\end{figure}

Using the soft thresholding operator as in \eqref{eq:picard_thresh} with fixed $\alpha_k = \alpha$ at all iterations, results in a modified fixed-point limit $\mathbf v_{\alpha} \in L^2(\Omega)$ defined by $\mathbf v_\alpha = \boldsymbol{\mathcal S}_\alpha \big( \mathbf \Phi_{\mathrm{Picard}} (\mathbf v_\alpha) \big)$,
as illustrated on the left of Figure \ref{fig:st_illustrations}. In particular, it does not coincide with $\mathbf v^{\ast}$ from \eqref{eq_fixed_point}, which we are however interested in.
This issue can be circumvented by choosing the soft thresholds such that $\lim_{k \rightarrow \infty} \alpha_k = 0$. An adaptive strategy presented and analyzed in \cite{BS17} consists in setting
\begin{equation*}
	\alpha_{k+1} = \begin{cases}
		\theta \alpha_{k}, \quad &\text{if } \ \| \boldsymbol{\mathcal S}_{\alpha_k} \big(  \mathbf \Phi_{\mathrm{Picard}} (\mathbf v^k) \big) -  \mathbf v^k \|_J \leq c \|  \mathbf \Phi_{\mathrm{Picard}}(\mathbf v^k) - \mathbf v^k \|_J, \\
		\alpha_k, \quad &\text{otherwise},
	\end{cases}
\end{equation*}
for some $c, \theta \in (0,1)$.
This amounts to decrease $\alpha$ whenever $\mathbf v^k$ gets too close to $\mathbf v_\alpha$, compared to its distance with $\mathbf v^*$,
see also Figure \ref{fig:st_illustrations} on the right for an illustration. The 
initial threshold $\alpha_0$ is taken large enough to ensure that
$\boldsymbol{\mathcal S}_{\alpha_0} \big( \mathbf \Phi_{\mathrm{Picard}}(\mathbf v^0) \big) = \boldsymbol{\mathcal S}_{\alpha_0} \big( \mathbf \Phi_{\mathrm{Picard}}(\mathbf 0) \big) = \mathbf 0$.
 
In this paper, however, we show that it is also possible to use a simpler strategy,
consisting in multiplying $\alpha$ by a fixed factor $\theta \in (0,1)$ at each iteration.
Using the same initialization, the thresholds are then given by $\alpha_k=\theta^k\alpha_0$.
As we will see in the following, this procedure yields both a convergence result and rank bounds on the iterates of the resulting fixed-point iteration.
The pseudo-codes for the
Picard iteration with constant decrease of $\alpha$ and SDC with adaptive decrease of $\alpha$
can be found in Algorithms \ref{alg:Picard} and \ref{alg:SDC}.
The other options, namely Picard iteration with adaptive decrease of $\alpha$ and SDC with constant decrease of $\alpha$, are obtained in a similar way.

\begin{algorithm}
	\caption{Gauss-Legendre-Picard iteration with constant decrease of $\alpha$}\label{alg:Picard}
	\begin{algorithmic}[1]
		\For{$n = 1, \dots, N$} \Comment{Time stepping }
		\State Adjust the nodes $t_1,\dots,t_J$ to the current interval $[(n-1)h, n h]$
		\State Initialize $k \gets -1$,  $\alpha_0 \gets \sigma_{\max}(v_0)$, $\mathbf v^0 \gets \mathbf 0$, $\mathrm{res} \gets 1$
		\While{$\mathrm{res} > \varepsilon_n$} \Comment{Fixed-point iteration}
		\State Increment $k \gets k+1$
		\State Compute $\boldsymbol \phi^k=\mathbf \Phi_{\mathrm{Picard}}(\mathbf v^k)$ from \eqref{twisted_fixed_point_iter}
		\State Update $\mathrm{res} \gets \|\boldsymbol \phi^k-\mathbf v^k\|_J$ 
		\State Compute $\mathbf v^{k+1} \gets \boldsymbol{\mathcal S}_{\alpha_k}( \boldsymbol \phi^k)$	\Comment{Soft thresholding}
		\State Update $\alpha_{k+1} \gets \theta \alpha_k$	\Comment{Fixed decrease of $\alpha$}
		\EndWhile
		\State $v_0 \gets  v_0 + \sum_{j = 1}^J \omega_j \mathcal F_{t_j} v^k_j$ \Comment{Approximation at subinterval boundary}
		\State $v_0 \gets \mathcal{R}_{\delta_n} (v_0)$ \Comment{Recompression to a tolerance $\delta_n$}
		
		\EndFor
	\end{algorithmic}
\end{algorithm}

\begin{algorithm}
	\caption{SDC with adaptive decrease of $\alpha$}\label{alg:SDC}
	\begin{algorithmic}[1]
		\For{$n = 1, \dots, N$} \Comment{Time stepping}
		
		\State Adjust the nodes $t_1,\dots,t_J$ to the current interval $[(n-1)h, nh]$
		\State Initialize $k \gets -1$,  $\alpha_0 \gets \sigma_{\max}(v_0)$, $\mathbf v^0 \gets \mathbf 0$, $\mathrm{res} \gets 1$
		\While{$\mathrm{res} > \varepsilon_n$} \Comment{Fixed point iteration}
		\State Increment $k \gets k+1$
		
		\For{$j = 1,\dots,J$}
		\State Compute $\phi_j^{k}$ from \eqref{eq_SDC_phi}
		\State Threshold $v^{k+1}_j \gets \mathcal S_{\alpha_k}( \phi^{k}_j )$ \Comment{Soft thresholding}
		\EndFor
		\State Update $\mathrm{res} \gets \|\boldsymbol \phi^{k}-\mathbf v^k\|_J$ 
		\State Compute $\mathrm{err} \gets \| \mathbf v^{k+1} - \mathbf v^k \|_J$
		
		\If{$\mathrm{err} \leq c \ \mathrm{res} $}	\Comment{Adaptive decrease of $\alpha$}
		\State $\alpha_{k+1} \gets \theta \alpha_k$
		\Else
		\State $\alpha_{k+1} \gets \alpha_k$
		\EndIf
		\EndWhile
		
		\State $v_0 \gets v_0 + \sum_{j = 1}^J \omega_j \mathcal F_{t_j} v^k_j$ \Comment{Approximation at subinterval boundary}
		\State $v_0 \gets \mathcal{R}_{\delta_n} (v_0)$ \Comment{Recompression to a tolerance $\delta_n$}
		\EndFor
	\end{algorithmic}
\end{algorithm}

In practice, we apply a recompression operator $\mathcal R_{\delta}$ of the form \eqref{eq:recompression}, with $\delta=\delta_{\mathrm{rel}}\,\mathrm{res}$, after every addition of two elements in $L^2(\Omega_1 \times \Omega_2)$. 
Here $\mathrm{res}$ denotes the current residual error estimator, and $\delta_{\mathrm{rel}}$ is a small parameter ensuring a bound on the relative error.
In the case of SDC, we need to use an additional recompression operator with a smaller relative tolerance in order to approximate the residual occuring implicitly in line 7 of the algorithm accurately enough. Here, the recompression tolerance is chosen relative to the current residual norm. See also \cite{DGR00} for a derivation of SDC based on residual and error functions using the spectral integration matrix.

\subsection{Quasi-optimal ranks}\label{sec:quasiopt}

Given $v \in L^2(\Omega)$ and $\varepsilon >0$, the low-rank best approximation of $v$ satisfies
\begin{equation}
\label{charac_Rdelta}
	\mathcal{R}_\varepsilon (v)\in \argmin_{\|v-w\|\leq \eps} \rank(w).
\end{equation}

However, the functions we want to approximate are not exactly known, 
hence we cannot hope to compute this optimal
low-rank representation.
We thus aim for a relaxed notion of ranks,
where we replace the constraint $\|v-w\|\leq \eps$ by a penalization of $\|v-w\|$.
Soft thresholding is naturally suited for this purpose,
as illustrated by the following result.

\begin{prop} \label{lm:qo_rank}
	Let $v \in L^2(\Omega)$ and $\alpha > 0$. It then holds that
	\begin{equation*}
		\| v - \mathcal S_\alpha v \|^2 = \min_{w \in L^2(\Omega)} \left\{ \| v - w \|^2 + \alpha^2 \rank(w) \right\}.
	\end{equation*}
\end{prop}

\begin{proof}
For any $w \in L^2(\Omega)$, using the Mirsky inequality from Lemma \ref{mirsky}, one can estimate
	\begin{align*}
		\| v -w \|^2 +\alpha^2\rank(w) 
		& \geq \sum_{k\in\N} |\sigma_k(v) - \sigma_k(w)|^2 +\alpha^2\mathds 1_{\sigma_k(w)\neq 0}\\
		&\geq \sum_{k\in\N} \min(\sigma_k(v)^2, \alpha^2)=\| v - \mathcal S_\alpha v \|^2,
	\end{align*}
	where we distinguished the cases $\sigma_k(w)=0$ and $\sigma_k(w)\neq 0$ to obtain the second line. Equality is attained for $w=\mathcal H_\alpha(u)$.
\end{proof}

Given $v \in L^2(\Omega)$ and $\eps \in (0,\|v\|)$, there exists $\alpha>0$ such that $\varepsilon = \norm{ v - \mathcal{S}_\alpha(v)}$. In view of Proposition \ref{lm:qo_rank} we will use the quantity 
\begin{equation}\label{eq:quasioptsubst}
    \frac{\norm{ v - \mathcal{S}_\alpha(v)}^2}{\alpha^2}=\frac{\eps^2}{\alpha^2} 
\end{equation}
as a benchmark for the ranks of an approximation of $v$.

\begin{remark}
When the singular values of $v$ decay regularly, these quantities correspond to quasi-optimal ranks, that is, they are bounded by the optimal ranks up to a constant. Indeed, it is shown in \cite[Prop.~3.6]{BS17} that
\[
  \rank(\mathcal{R}_\varepsilon(v)) \leq C \varepsilon^{-\frac1s} \quad \text{implies} \quad
    \frac{\norm{ v - \mathcal{S}_\alpha(v)}^2}{\alpha^2}  \leq  C_s \varepsilon^{-\frac1s},
\]
where $C_s>0$ depends on $C>0$ and $s>0$. Moreover
\[
  \rank(\mathcal{R}_\varepsilon(v)) \leq c \bigl( 1+ \abs{\log \varepsilon} \bigr)^{\beta} \quad \text{implies} \quad
    \frac{\norm{ v - \mathcal{S}_\alpha(v)}^2}{\alpha^2}  \leq  C_{c,\beta} \bigl( 1 +  \abs{\log \varepsilon} \bigr)^{\beta},
 \]
with $C_{c,\beta}>0$ depending on $c$ and $\beta$. Thus in both cases,
\eqref{eq:quasioptsubst} yields quasi-optimal ranks.
\end{remark}

\section{Analysis of the Picard iteration}\label{sec:analysis}

In this section, we analyze the accuracy of the proposed Gauss-Legendre-Picard iteration scheme and provide bounds on the ranks of the approximations at both intermediate time steps $t_1, t_2, \dots, t_J \in ((n-1)h, nh)$ and subintervals boundaries $nh$ for  $1 \leq n \leq N$. The resulting estimates can then be combined together to a global quasi-optimality result.

For notational convenience, we introduce the following simplified notation: since we restrict our considerations to the case of the Gauss-Legendre-Picard iteration, we denote by $ \Phi$ and $\mathbf\Phi$ the equivalent fixed-point iterations $ \Phi_{\mathrm{Picard}}$ and $\mathbf\Phi_{\mathrm{Picard}}$ but formulated in untwisted variables, that is,
\[
\Phi\mathbf u(t)= e^{it\Delta} u_0 - i\int_0^t\sum_{m=1}^J e^{i(t-t_m)\Delta}Vu_m\,\ell_m(s)\,ds
\]
and
\begin{equation}
\label{eq_untwisted_phi}
(\mathbf \Phi \mathbf u)_j = \Phi \mathbf u(t_j)= e^{it_j\Delta} u_0 - i\sum_{m=1}^J \omega_{j,m}e^{i(t_j-t_m)\Delta}Vu_m
\end{equation}
for $\mathbf u \in L^2(\Omega)^J$.
Note that the fixed-point iterations on the original and untwisted variable have both the same contraction constant $\rho$. We also omit the parentheses when applying the (nonlinear) operators $\Phi$, $\mathcal S_\alpha$, $\mathcal H_\alpha$ and $\mathcal R_\delta$.

\begin{figure}[h] 
	\centering
	\begin{minipage}[t]{\linewidth}
		\hspace*{0.25\linewidth}
	\begin{tikzpicture}[scale=1, font=\small]
		\node[label=left:{$u_0^\diamond$}] at (0,1) {$\bullet$};
		\node[label=left:{$u_0^\ast$}] at (0,0) {$\bullet$};
		
		\node[label=right:{$u(h)$, exact evolution}] at (4,2.6) {$\bullet$};
		\node[blue,label=right:{$u^\diamond(h)$, collocation solution}] at (4,1.7) {$\bullet$};
		\draw[->,thick] (0,1) to[bend left=20] (4,2.6);
		\draw[blue,->,thick,dashed] (0,1) to[bend left=10] (4,1.7);
		
		\node[blue,label=right:{$u^\ast(h)$, collocation solution}] at (4,0.7) {$\bullet$};
		\node[red,label=right:{$\tilde{u}(h) = \mathcal R_\delta \Phi \mathbf u^K(h)$, recompresed scheme}] at (4,0.1) {$\bullet$};
		\draw[blue,->,thick,dashed] (0,0) to[bend left=10] (4,0.7);
		\draw[red,->,thick,dotted] (0,0) to[bend right=15] (4,0.1);
		
		\draw[->] (-1,-1) -- (5,-1) node [pos=1.05] {$t$};
	\end{tikzpicture}
	\end{minipage}
	\caption{Visualization of the decomposition of the local error in Proposition~\ref{prop_local_error}.}
	\label{fig_local_error_decomposition}
\end{figure}
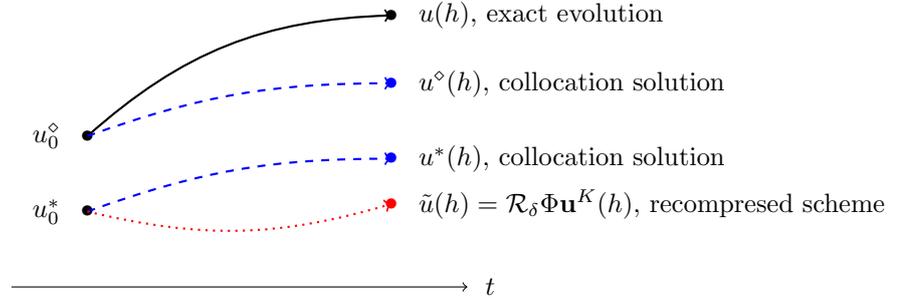

Let $u^{\ast}(t)$ be the exact solution of the collocation formulation \eqref{collocation} on $[0,h]$ with initial data $u_0^*$, that is,
\begin{equation*}
u^{\ast}(t) = e^{it \Delta} u_0^{\ast} - i \int_{0}^{t} \sum_{m=1}^{J} e^{i(t - t_m) \Delta} V u^{\ast}_m \ell_m(s) \, ds, \quad \text{for} \quad t \in [0,h],
\end{equation*}
with $u^*_m=u^*(t_m)$, and let $u^{\diamond}(t)$ be the solution of the collocation formulation on $[0,h]$ with initial data $u_0^\diamond$. We further write $u(t)$ for the exact solution of the Schrödinger equation with initial data $u^{\diamond}_0$ and $\tilde{u}(t)$ for the approximate solution produced by the scheme initialized at $u_0^\ast$, see Figure~\ref{fig_local_error_decomposition}.

The following analysis requires some smoothness assumptions on the potential $V$. We therefore assume that the potential is uniformly bounded, that is,  $C_V < \infty$. Moreover, we assume that the quantities
\begin{equation*}
	\|(\partial_t-i\Delta+iV)^{2J+1}u^\diamond \|_{\mathcal W(0,h)} \quad \text{and} \quad \|\partial_t^{J} (e^{-it \Delta}V u) \|_{\mathcal W(0,h)}
\end{equation*}
can also be bounded from above. In order to derive rank bounds for the iterates of the considered algorithm, we also assume that the potential function has finite rank
\begin{equation*}
	V = \sum_{\ell= 1}^{\rank(V)} V_1^{(\ell)} \otimes V_2^{(\ell)}, \quad \text{for} \quad V_i^{(\ell)} \in L^2(\Omega_i), \ i = 1,2,
\end{equation*}
with $\rank(V) \in \mathbb N$. We emphasize that this is the rank of the function $(x_1,x_2)\mapsto V(x_1,x_2)$, which differs from the (potentially infinite) rank of the multiplication operator $u\mapsto Vu$.

\subsection{Local accuracy bound}

Since we aim at approximating the solution of the Schrödinger equation in the untwisted variable, we formulate the accuracy results accordingly in terms of the original variable. However, we still make use of the iteration \eqref{twisted_picard} in the proofs. Note that the accuracy results in this section hold independently of the particular strategy for decreasing the soft thresholding parameter $\alpha$.

The following theorem provides a local accuracy bound for the analyzed scheme. Note that here, if iterated over the subintervals, the bound will depend only linearly on the number of needed subintervals, which would correspond to a linear error accumulation and therefore achieve the optimal rate one can hope for in this setting.

\begin{prop} \label{prop_local_error}
Let $\delta$ be the recompression parameter for the approximation at the subinterval boundary and $\varepsilon$ the tolerance of the Gauss-Legendre-Picard iteration. The scheme then achieves an accuracy of
\[
\|u(h)-\tilde u(h)\|\leq \delta+\frac{\rho\,\eps}{1-\rho}+\|u_0^\diamond-u_0^*\|+\kappa_{2J}h^{2J+1},
\]
where we recall that $\rho = \Lambda_J C_V h$ denotes the contraction constant and 
\begin{equation}
	\kappa_{2J} = \frac{(J!)^4}{(2J+1)!(2J)!^2} \|(\partial_t-i\Delta+iV)^{2J+1}u^\diamond \|_{\mathcal W(0,T)}. \label{kappa2J}
\end{equation}
\end{prop}

\begin{remark}
By developing the term $(\partial_t-i\Delta+iV)^{2J+1}u^\diamond$, and using commutator identities, one can see that the constant $\kappa_{2J}$ involves $4J$ derivatives of $V$, but only $2J$ derivatives of~$u^\diamond$, which translate into $2J$ derivatives on $u$. This gain in regularity requirements based on commutator estimates is already observed in \cite{AB23}, where it is also shown to hold when the equation contains the nonlinearity~$|u|^2u$.
\end{remark}

The proof of the above proposition will make use of two auxiliary lemmas. For simplicity of notation, we formulate them for the first interval $[0,h]$. However, they also analogously hold on each of the other subintervals $[nh,(n+1)h]$. First of all, we note that the proposed scheme is norm preserving.

\begin{lemma}[Isometry preservation] \label{lm:isometry}
	It holds that $\|u^*(h)-{u}^\diamond(h)\|=\|u_0^{\ast} - u_0^\diamond\|$.
\end{lemma}

Note that this result in fact only holds for the fixed-point solutions, and not for the intermediate iterations.

\begin{proof}
According to \eqref{collocation}, in the twisted variables, $p(t)= e^{-it \Delta} (u^{\ast}(t)- u^\diamond(t))$ is a polynomial of degree $J$, which satisfies the collocation identity
\[
\partial_t{p}(t_j) = \mathcal{F}_{t_j}(p(t_j))=-ie^{-it_j\Delta}Ve^{it_j\Delta}p(t_j),\qquad 1\leq j\leq J.
\]
As $ \langle \partial_t{p}(t), p(t) \rangle$ is a polynomial of degree $2J-1$, its Gauss quadrature is exact, hence
	\begin{align*}
		\|u^*(h)-u^\diamond(h)\|^2 -\|u_0^* - u_0^\diamond\|^2
		&= \| p(h) \|^2 - \| p(0) \|^2 = \int_0^h\frac{d}{dt}\|p(s)\|^2 \, ds \\
		&=  2\Re \int_0^h \langle \partial_t{p}(s), p(s) \rangle \, ds
		=2\Re \sum_{j=1}^J \omega_j \langle\partial_t{p}(t_j), p(t_j) \rangle\\
		&= 2\Re  \Big(-i \sum_{j=1}^J \|e^{it_j \Delta}p(t_j) \|^2\Big)=0. \qedhere
\end{align*}
\end{proof}

Next, we prove that the fixed-point is able to reach a local accuracy of order $2J+1$ at the subinterval boundary, which can then be translated into a global error bound of order $2J$.

\begin{lemma}[Local error at subinterval boundary] \label{lm:locerr2J}
	It holds that
	\begin{equation*}
		\|u^\diamond(h)-u(h)\| \leq \kappa_{2J}  h^{2J+1},
	\end{equation*}
	where the constant $\kappa_{2J}$ is defined as in \eqref{kappa2J}.
\end{lemma}

\begin{proof}
	
	Let $e^{it(\Delta-V)}$ be the flow map of the Schrödinger equation, which is well defined on $L^2(\Omega)$. For $t\in [0,h]$, setting
	\begin{equation*}
		f(t)=e^{i(h-t)(\Delta-V)}u^\diamond(t)
	\end{equation*}
	 yields
	\begin{equation} \label{eq:temp2}
		u^\diamond(h)-u(h)=f(h)-f(0)=\int_0^h \partial_tf(t)\,dt
	\end{equation}
	where
	\begin{equation*}
		\partial_tf(t)=e^{i(h-t)(\Delta-V)}(\partial_t-i\Delta+iV)u^\diamond(t).
	\end{equation*}
	
	Observe in particular that $\partial_tf(t_j)=0$ since $u^{\ast}(t)$ is the solution of a collocation method. We can thus bound \eqref{eq:temp2} by
	\begin{equation*}
		\|u^\diamond(h)-u(h)\| = \biggnorm{ \int_0^h \partial_tf(t)\,dt-\sum_{j=1}^J \omega_j\partial_tf(t_j)} \leq  \frac{(J!)^4  }{(2J+1)!(2J)!^2}\|\partial_t^{2J+1}f\|_{\mathcal W(0,h)}h^{2J+1}.
	\end{equation*}
	The last bound is a standard accuracy result for Gauss-Legendre quadrature, and can be found for example in \cite{stoer_introduction_1980}. To make sure that it also works in the Hilbert-space-valued case, we provide a proof with explicit constants in Lemma~\ref{quadrature_has_order_2J} of the appendix.
	
	The statement of the lemma then follows from
	\begin{align*}
		\|\partial_t^{2J+1}f\|_{\mathcal W(0,h)} &= \max_{t \in [0,h]} \|e^{i(h-t)(\Delta-V)}(\partial_t-i\Delta+iV)^{2J+1}u^\diamond(t)\| \\
		&= \|(\partial_t-i\Delta+iV)^{2J+1}u^\diamond\|_{\mathcal W(0,h)}. \qedhere
	\end{align*}
\end{proof}

Inside the interval $[0,h]$, we obtain a local accuracy bound of order $J+1$.
\begin{lemma}[Local error bound inside subintervals]
\label{localerrorboundinside}
	It holds that
	\begin{equation*}
		\|u^*(t) - u(t)\| \leq \frac{\kappa_J}{1-\rho}h^{J+1}+ \frac{1}{1- \rho} \|u_0^* - u^\diamond_0\|,
	\end{equation*}
	for all $t \in [0,h]$, where
	\[
	\kappa_J=\frac{J!}{\sqrt{2J+1}(2J)!}\|\partial_t^{J} (e^{-it \Delta}V u) \|_{\mathcal W(0,h)}.
	\]
\end{lemma}

\begin{proof}
	For the twisted variables $v^{\ast}(t) = e^{-it \Delta} u^{\ast}(t)$ and  $v(t) = e^{-it \Delta} u(t)$ it holds that
	\begin{align*}
		\| u^{\ast}(t) - u(t) \| &\leq \| u^\ast_0 - u^\diamond_0 \| +  \biggnorm{ \int_{0}^{t} \sum_{m=1}^{J} \mathcal{F}_{t_m}(v_m^{\ast} - v(t_m))  \ell_m(s) \, ds } \\
		&\quad + \biggnorm{ \int_{0}^{t} \Big( \sum_{m=1}^{J} \mathcal{F}_{t_m}  v(t_m) \ell_m(s) - \mathcal{F}_s v(s) \Big) \, ds }.
	\end{align*}
	By Proposition~\ref{prop_contractivity}, the second term is bounded by $\rho\|v^*-v\|_{\mathcal W(0,h)}=\rho\|u^*-u\|_{\mathcal W(0,h)}$.\\
	Applying Lemma~\ref{lem_inner_quadrature} to $g(s) = \mathcal{F}_s v(s)=-ie^{is\Delta}Vu(s)$, we bound the third term by $\kappa_J h^{J+1}$.
	We conclude by taking the supremum over all $t\in [0,h]$, and passing the second term to the left hand side.
\end{proof}

We are finally in the position to prove our local accuracy bound.

\begin{proof}[Proof of Proposition \ref{prop_local_error}]
	Let $K$ denote the value of $k$ at the end of Algorithm~\ref{alg:Picard}, for which we know that
	\begin{equation*}
		\| \mathbf \Phi \mathbf u^{K} - \mathbf u^{K} \|_J \leq \varepsilon.
	\end{equation*}
	We can then decompose the error according to
	\begin{equation*}
		\|u(h)-\tilde u(h)\|
		\leq \|u(h)-u^\diamond(h)\|+\|u^\diamond(h)-u^*(h)\|+\|u^*(h)-\Phi \mathbf u^K(h)\|+\|\Phi \mathbf u^K(h)-\tilde u(h)\|.
	\end{equation*}
	
	The first term can be bounded by $ \kappa_{2J} h^{2J+1}$ using Lemma \ref{lm:isometry}, and the second one is equal to $ \| u^\diamond_0 - u^\ast_0 \|$ according to Lemma~\ref{lm:locerr2J}.
For the third term, Proposition~\ref{prop_contractivity} yields
\[
\|u^*-\Phi \mathbf u^K\|_{\mathcal W(0,h)}\leq \rho \|\mathbf u^*-\mathbf u^K\|_J\leq \rho\big(\|u^*-\Phi \mathbf u^K\|_{\mathcal W(0,h)}+\|\mathbf u^K-\mathbf \Phi \mathbf u^{K}\|_J\big),
\]
and thus
\[
\|u^*(h)-\Phi \mathbf u^K(h)\|\leq \|u^*-\Phi \mathbf u^K\|_{\mathcal W(0,h)}\leq \frac{\rho}{1-\rho}\|\mathbf u^K-\mathbf \Phi \mathbf u^{K}\|_J\leq \frac{\rho\eps}{1-\rho}.
\]
Finally, by the characterisation \eqref{charac_Rdelta} of $\mathcal R_\delta$, the last term is bounded by the recompression tolerance $\delta$.
\end{proof}

\subsection{Local rank bounds}

Next, we proceed to the derivation of local rank bounds for the iterates of Algorithm \ref{alg:Picard} with constant decrease of $\alpha_k$, motivated by the notion of quasi-optimality discussed in Section \ref{sec:quasiopt}.
As introduced in the previous proof, let $K$ denote the value of $k$ at the end of Algorithm~\ref{alg:Picard}. Additionally, we set
\begin{equation*}
	\mathbf u^\ast = (u^\ast(t_m))_{1 \leq m \leq J} \quad\text{and}\quad \mathbf u = (u(t_m))_{1 \leq m \leq J}.
\end{equation*}

The following lemma shows that the Gauss-Legendre-Picard iterations produce quasi-optimal low-rank approximations. It simplifies the analysis from \cite{BS17} by allowing a constant decrease of $\alpha$.

\begin{lemma} \label{lem_rank_uk}
Starting from the initial data $u_0^*$, the approximations generated by Algorithm~\ref{alg:Picard} in untwisted variables satisfy
\[
\| \mathbf u^{k+1}- \mathbf u^\ast \|_J \leq \frac{2}{1-\rho}\| \mathbf u^\ast - \boldsymbol{\mathcal S}_{\alpha_k} \mathbf u^*\|_J, \quad \text{for all} \quad 0 \leq k \leq K.
\]
\end{lemma}
\begin{proof}
Since $\boldsymbol{\mathcal S}_{\alpha_0} \mathbf \Phi$ is contractive with fixed point at $\mathbf 0$, it holds that
\[
\|\mathbf u^*\|_J - \| \mathbf u^* - \boldsymbol{\mathcal S}_{\alpha_0} \mathbf u^*\|_J \leq \| \boldsymbol{\mathcal S}_{\alpha_0} \mathbf u^*\|_J = \| \boldsymbol{\mathcal S}_{\alpha_0} \mathbf \Phi \mathbf u^*\|_J \leq \rho \|\mathbf u^*\|_J
\]
and thus
\[
\| \mathbf u^0 - \mathbf u^*\|_J = \| \mathbf u^*\|_J \leq \frac{1}{1-\rho}\| \mathbf u^* - \boldsymbol{\mathcal S}_{\alpha_0} \mathbf u^*\|_J.
\]
We then proceed by induction on $k$. For $k\geq 0$, we have
\[
\| \mathbf u^{k+1} - \mathbf u^*\|_J
\leq \| \boldsymbol{\mathcal S}_{\alpha_k} \mathbf \Phi \mathbf u^{k} - \boldsymbol{\mathcal S}_{\alpha_k} \mathbf \Phi \mathbf u^* \|_J + \| \mathbf u^* - \boldsymbol{\mathcal S}_{\alpha_k} \mathbf \Phi \mathbf u^*\|_J
\leq \rho\| \mathbf u^k - \mathbf u^* \|_J + \| \mathbf u^* - \boldsymbol{\mathcal S}_{\alpha_k} \mathbf u^* \|_J.
\]
By Lemma~\ref{lemma:thresholds}, for $0\leq \ell\leq k$, it holds that
\[
\theta^\ell \| \mathbf u^* - \boldsymbol{\mathcal S}_{\alpha_{k-\ell}} \mathbf u^*\|_J \leq \| \mathbf u^* - \boldsymbol{\mathcal S}_{\alpha_{k}} \mathbf u^*\|_J.
\]
By induction, we then obtain
\begin{align*}
\| \mathbf u^{k+1} - \mathbf u^* \|_J
&\leq \sum_{\ell=0}^k \rho^\ell \| \mathbf u^* - \boldsymbol{\mathcal S}_{\alpha_{k-\ell}} \mathbf u^* \|_J + \rho^{k+1} \| \mathbf u^0 - \mathbf u^* \|_J \\
&\leq \| \mathbf u^* - \boldsymbol{\mathcal S}_{\alpha_k} \mathbf u^*\|_J \biggl( \sum_{\ell=0}^k \frac{\rho^\ell}{\theta^\ell}+ \frac{\rho^{k+1}}{\theta^k} \frac{1}{1-\rho}\biggr).
\end{align*}
Setting $\theta=\sqrt\rho$, the last constant can be bounded by
\[
\sum_{\ell=0}^{k-1} \frac{\rho^\ell}{\theta^\ell}+ \frac{\rho^{k}}{\theta^k} \left(1+\frac{\rho}{1-\rho}\right)
=\frac{1-\sqrt \rho^k}{1-\sqrt\rho}+\frac{\sqrt\rho^k}{1-\rho}\leq \frac{2}{1-\rho}. \qedhere
\]
\end{proof}

\begin{remark}
It is in fact possible to take an arbitrary $\theta\in [\rho,1)$ in the proof. For $\theta=\rho$, we achieve the optimal rate of decay of $(\alpha_k)_{k \geq 0}$, and the constant in the error bound is at most $k + \frac{1}{1-\rho}$.
For $\theta=\sqrt\rho$, the algorithm needs twice as many iterations in order to reach a desired value of the threshold, but one obtains a constant independent of the iteration number $k$.
\end{remark}

We also need the following lemma, which was originally proposed in \cite[Lemma 4.3]{BS17} and allows to compare the ranks of the scheme to the ranks of the exact solution.
\begin{lemma}
\label{lemma-rank-rank-error}
For any $u,v\in L^2(\Omega)$ and $\alpha,\beta>0$, it holds that
\[
\rank(\mathcal S_{\alpha+\beta}(u))\leq \rank(\mathcal S_{\alpha}v)+\frac{1}{\beta^2}\|u-v\|^2.
\]
\end{lemma}

With this, we obtain the quasi-optimality of the ranks of the iterates $\mathbf u^k$.
\begin{lemma}
\label{lemmaboundingtherankofuk}
For $\eps \geq 4 \, \frac{1+ \rho}{1-\rho} \, \| \mathbf u - \mathbf u^*\|_J$, let $\bar \alpha$ be the value of the soft threshold such that
\[
\| u - \mathcal S_{\bar \alpha} u \|_{\mathcal W(0,h)} = \frac{1-\rho}{1 + \rho}\, \frac{\eps}{4}.
\]
Then for $0\leq k\leq K$, 
\begin{equation}
	\rank (\mathbf u^{k}) := \max_{1 \leq j \leq J} \rank(u^k_j) \leq \bar r:=\frac{\eps^2}{\rho \bar \alpha^2}
	=\frac{16}{\rho}\frac{(1+\rho)^2}{(1-\rho)^2}\,\frac{\|  u - \mathcal S_{\bar \alpha} u \|_{\mathcal W(0,h)}^2}{\bar \alpha^2}.
\label{eq:r_bar}
\end{equation}
\end{lemma}

\begin{proof}
	The value of $\bar{\alpha}(t)$ is well-defined since the soft thresholding error is continuous and monotonic.
The ranks of the first two iterates $\mathbf u^0$ and $\mathbf u^1$ are simply zero. For the remaining iterations, we let $0\leq k \leq K-2$ and apply Lemma~\ref{lemma-rank-rank-error} to obtain
\[
\rank (\mathbf u^{k+2}) = \rank (\boldsymbol{\mathcal S}_{\alpha_{k+1}} \mathbf \Phi \mathbf u^{k+1}) \leq \rank ( \boldsymbol{\mathcal S}_{\alpha_{k+1}/2} \mathbf u) + \frac{4}{\alpha_{k+1}^2} \| \mathbf \Phi \mathbf u^{k+1} - \mathbf u \|_J^2.
\]
The first term is controlled by
\begin{equation}
\label{rankofSalphakplusoneovertwo}
\rank (\boldsymbol{\mathcal S}_{\alpha_{k+1}/2} \mathbf u) \leq \frac{4}{\alpha_{k+1}^2} \| \mathbf u - \boldsymbol{\mathcal S}_{\alpha_{k+1}/2} \mathbf u \|_J^2 \leq \frac{4}{\alpha_{k+1}^2} \| \mathbf u - \boldsymbol{\mathcal S}_{\alpha_k} \mathbf u \|_J^2.
\end{equation}
For the second term, notice that
\[
\| \mathbf \Phi \mathbf u^{k+1} - \mathbf u \|_J \leq \| \mathbf \Phi \mathbf u^{k+1} - \mathbf u^* \|_J + \| \mathbf u^* - \mathbf u \|_J \leq \rho \| \mathbf u^{k+1} - \mathbf u^* \|_J + \| \mathbf u^* - \mathbf u \|_J.
\]
Using Lemma~\ref{lem_rank_uk} and the non-expansivity of $\mathcal I-\mathcal S_{\alpha_k}$, we obtain that
\[
\| \mathbf u^{k+1} - \mathbf u^* \|_J \leq \frac{2}{1-\rho} \| \mathbf u^* - \boldsymbol{\mathcal S}_{\alpha_k} \mathbf u^*\|_J \leq \frac{2}{1-\rho} \left(\| \mathbf u - \boldsymbol{\mathcal S}_{\alpha_k} \mathbf u \|_J +\| \mathbf u^* - \mathbf u \|_J\right),
\]
which yields
\begin{equation}
\label{normofPhiukplusoneminusu}
\begin{aligned}
\| \mathbf \Phi \mathbf u^{k+1} - \mathbf u\|_J^2
&\leq \biggl(\frac{2\rho}{1-\rho}\| \mathbf u - \boldsymbol{\mathcal S}_{\alpha_k} \mathbf u \|_J + \frac{1+\rho}{1-\rho} \| \mathbf u^* - \mathbf u \|_J\biggr)^2 \\
&\leq \frac{8\rho^2}{(1-\rho)^2}\| \mathbf u - \boldsymbol{\mathcal S}_{\alpha_k} \mathbf u \|_J^2 + 2\,\frac{(1+\rho)^2}{(1-\rho)^2} \| \mathbf u^* - \mathbf u \|_J^2,
\end{aligned}
\end{equation}
where we used the inequality $(a+b)^2\leq 2a^2+2b^2$ to obtain the second line.

Combining equations \eqref{rankofSalphakplusoneovertwo} and \eqref{normofPhiukplusoneminusu}, we get
\begin{equation}
\label{boundbyatriplefraction}
\begin{aligned}
\rank ( \mathbf u^{k+2})
&\leq \frac{4}{\alpha_{k+1}^2} \biggl( \frac{(1-\rho)^2+8\rho^2}{(1-\rho)^2}\| \mathbf u - \boldsymbol{\mathcal S}_{\alpha_k} \mathbf u \|_J^2 + 2\,\frac{(1+\rho)^2}{(1-\rho)^2} \| \mathbf u^* - \mathbf u \|_J^2\biggr)\\
&\leq \frac{8}{\rho}\,\frac{(1+\rho)^2}{(1-\rho)^2}\,\frac{\| \mathbf u - \boldsymbol{\mathcal S}_{\alpha_k} \mathbf u\|_J^2 + \| \mathbf u^* - \mathbf u \|_J^2}{\alpha_k^2},
\end{aligned}
\end{equation}
where the second line comes from the fact that $\alpha_{k+1}^2=\rho \alpha_k^2$, together with the inequality $(1-\rho)^2+8\rho^2\leq 2(1+\rho)^2$, which is true since $\rho^2\leq \rho\leq 1$.

It remains to replace $\alpha_k$ by $\bar \alpha$. As $k+1<K$, the stopping criterion implies that
\[
\frac{\eps}{1 + \rho} \leq \| \mathbf u^{k+1} - \mathbf u^*\|_J \leq \frac{2}{1-\rho} \| \mathbf u^* - \boldsymbol{\mathcal S}_{\alpha_k} \mathbf u^* \|_J,
\]
where we again make use of Lemma~\ref{lem_rank_uk}. Using the condition on $\eps$ and the non-expansivity of $\mathcal I - \mathcal S_{\alpha_k}$, we obtain that
\begin{align*}
\| u - \mathcal S_{\alpha_k} u \|_{\mathcal W(0,h)} &\geq \| \mathbf u - \boldsymbol{\mathcal S}_{\alpha_k} \mathbf u \|_J \geq \| \mathbf u^* - \boldsymbol{\mathcal S}_{\alpha_k} \mathbf u^* \|_J - \| \mathbf u - \mathbf u^* \|_J \\ &\geq \frac{1-\rho}{1 + \rho}\, \frac{\eps}{2} - \| \mathbf u - \mathbf u^* \|_J \geq \frac{1-\rho}{1 + \rho} \, \frac{\eps}{4} = \| u - \mathcal S_{\bar \alpha} u \|_{\mathcal W(0,h)}.
\end{align*}
This shows that $\alpha_k\geq \bar \alpha$, thus
\[
\frac{\| \mathbf u - \boldsymbol{\mathcal S}_{\alpha_k} \mathbf u \|_J^2}{\alpha_{k}^2}\leq \frac{\| \mathbf u - \boldsymbol{\mathcal S}_{\bar \alpha} \mathbf u \|_J^2}{\bar \alpha^2} \quad \text{and} \quad \frac{\| \mathbf u^* - \mathbf u\|_J^2}{\alpha_k^2} \leq \frac{\| \mathbf u^* - \mathbf u\|_J^2}{\bar \alpha^2}.
\]
As a consequence, we can replace $\alpha_k$ by $\bar \alpha$ in equation~\eqref{boundbyatriplefraction}, and conclude by observing that
\[
\| \mathbf u^* - \mathbf u\|_J \leq \frac{1-\rho}{1 + \rho} \, \frac{\eps}{4} = \| u - \mathcal S_{\bar \alpha} u\|_{\mathcal W(0,h)} \quad \text{and} \quad \| \mathbf u - \boldsymbol{ \mathcal S}_{\bar \alpha} \mathbf u \|_J \leq \| u - \mathcal S_{\bar{\alpha}} u \|_{\mathcal W(0,h)}. \qedhere
\]
\end{proof}

As discussed in Section \ref{sec:quasiopt}, the quantity $\|  u - \mathcal S_{\bar \alpha} u \|_{\mathcal W (0,h)}^2/\bar \alpha^2$ can indeed be considered to describe a quasi-optimal rank for the corresponding error. Note that the factor $1/ \rho$ in the bound \eqref{eq:r_bar} can be improved by choosing  $\theta$ closer to $1$.

We proceed in the derivation of a local rank bound by considering how the application of the fixed-point operator modifies the rank of the current iterate.

\begin{lemma}
\label{lemmarankofPhiuk}
For $0\leq k\leq K$ the ranks of the fixed-point iteration $\mathbf \Phi$ applied to the current iterate $\mathbf u^k$ can be bounded in terms of
\[
\rank(\mathbf \Phi \mathbf u^k) \leq \rank(u^\ast_0)+J \rank(V) \rank(\mathbf u^k).
\]
\end{lemma}

\begin{proof}
It suffices to use the estimates $\rank(f+g)\leq \rank(f)+\rank(g)$ and $\rank(fg)\leq \rank(f)\rank(g)$ together with the definition of the fixed-point formulation.
\end{proof}

Finally, we need an estimate for the ranks after the final recompression step carried out by applying $\mathcal R_\delta$ as defined in \eqref{eq:recompression}. This estimate is illustrated in Figure~\ref{fig:tikzofrecompressionranks}.

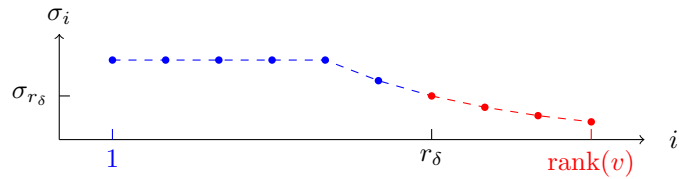
\begin{figure}[h]
	\centering
	\begin{tikzpicture}[scale=0.7, font=\small]
		\draw[->] (0,0) -- (11,0) node [pos=1.05] {$i$};
		\draw[->] (0,0) -- (0,2) node[anchor=south] {$\sigma_i$};
		
		\pgfmathsetmacro{\valThird}{5*exp(-0.3*(3-1))}
		\pgfmathsetmacro{\valFourth}{5*exp(-0.3*(4-1))}
		\pgfmathsetmacro{\valFifth}{5*exp(-0.3*(5-1))}
		\pgfmathsetmacro{\vSix}{5*exp(-0.3*(6-1))}
		\pgfmathsetmacro{\vSeven}{5*exp(-0.3*(7-1))}
		\pgfmathsetmacro{\vEight}{5*exp(-0.3*(8-1))}
		\pgfmathsetmacro{\vNine}{5*exp(-0.3*(9-1))}
		\pgfmathsetmacro{\vTen}{5*exp(-0.3*(10-1))}
		
		\foreach \i in {1,2,3,4,5} {
			\fill[blue] (\i,\valFifth) circle (2pt);
		}
		\draw[blue, dashed]
		(1,\valFifth) -- (5,\valFifth) -- (6,\vSix) -- (7,\vSeven);
		\draw[red, dashed]
		(7,\vSeven)-- (8,\vEight) -- (9,\vNine) -- (10,\vTen);

		\pgfmathsetmacro{\val}{5*exp(-0.3*(6-1))}
		\fill[blue] (6,\val) circle (2pt);
		
		\foreach \i in {7,8,9,10} {
			\pgfmathsetmacro{\val}{5*exp(-0.3*(\i-1))}
			\fill[red] (\i,\val) circle (2pt);
		}
		
		\draw (0,\vSeven) -- (0.2,\vSeven); 
		\node[left] at (0,\vSeven) {$\sigma_{r_\delta}$};
		
		\draw[blue] (1,0) -- (1,0.2); 
		\node[below, blue] at (1,0) {$1$};
		
		\draw (7,0) -- (7,0.2); 
		\node[below] at (7,0) {$r_\delta$};
		
		\draw[red] (10,0) -- (10,0.2); 
		\node[below, red] at (10,0) {$\rank(v)$};
	\end{tikzpicture}
	\caption{Visualization of Lemma \ref{lemmarecompressionranks}: the squares of the red singular values sum up to at least $\delta^2$, and the blue ones to at most $\|v\|^2-\delta^2$, thus $r_\delta$ cannot be too close to $\rank(v)$.}
\label{fig:tikzofrecompressionranks}
\end{figure}

\begin{lemma}
\label{lemmarecompressionranks}
For any $v\in L^2(\Omega)$ of finite rank, it holds that
\[
r_\delta=\rank(\mathcal R_\delta v) \leq 1 + \frac{\| v \|^2-\delta^2}{\|v\|^2}\rank(v).
\]
\end{lemma}
\begin{proof}
Recall the definitions from \eqref{eq:recompression}:
\[
\mathcal{R}_\delta(v) =  \sum_{1 \leq k \leq r_{\delta}}  \sigma_k\, v^{(1)}_k \otimes v^{(2)}_k, \quad \text{where} \quad r_\delta = \min \Bigl\{ r \in \N_0 \colon  \sum_{k > r} \sigma_k^2 \leq \delta^2 \Bigr\}.
\]
By the optimality of $r_\delta$, we know that
\[
\delta^2<\sum_{k\geq r_\delta}\sigma_k^2,\quad\text{and}\quad \|v\|^2=\sum_{k=1}^{\rank(v)}\sigma_k^2.
\]
Therefore, as the singular values are non-increasing,
\[
\frac{\|v\|^2}{\|v\|^2-\delta^2}<\frac{\|v\|^2}{\|v\|^2-\sum_{k\geq r_\delta}\sigma_k^2}=1+\frac{\sum_{k\geq r_\delta}\sigma_k^2}{\sum_{k< r_\delta}\sigma_k^2}
\leq 1+\frac{(\rank(v)-r_\delta+1)\sigma_{r_\delta}^2}{(r_\delta-1)\sigma_{r_\delta}^2}=\frac{\rank(v)}{r_\delta-1}.\qedhere
\]
\end{proof}

In order to apply Lemma~\ref{lemmarecompressionranks}, we need $v$ to be of bounded rank and have a small $L^2(\Omega)$ norm. Instead of the natural choice $v=\Phi \mathbf u^K(t)$, we will take $v=\Phi \mathbf u^K(h)-\mathcal S_{\tilde \alpha}\Phi \mathbf u^K(h)$, which has smaller norm, but the same singular values below the threshold $\tilde \alpha$.
\begin{lemma}
\label{lemmarankofutilde}
For $\eps \geq 4\frac{1+\rho}{1-\rho} \| u - u^*\|_{\mathcal W(0,h)}$, the rank of
$\tilde u(h)=\mathcal R_\delta \Phi \mathbf u^K(h)$ can be bounded by
\[
\rank(\tilde u(h))\leq \max\biggl\{\frac{\bar r\rho}{2(1-\rho)^2},1+\frac{\eps^2-\delta^2(1-\rho)^2}{\eps^2}\rank(\Phi \mathbf u^K(h))\biggr\},
\]
where $\bar{r}$ is defined as in \eqref{eq:r_bar}.
\end{lemma}

\begin{proof}
Let $\tilde \alpha$ be the threshold value such that
\[
\|\Phi \mathbf u^K(h)-\mathcal S_{\tilde \alpha}\Phi \mathbf u^K(h) \|=\frac{\eps}{1-\rho}.
\]
As
\[
\| \Phi \mathbf u^K(h) - u(h) \| \leq \frac{\rho\,\eps}{1-\rho}+\| u(h) - u^*(h)\| \leq \eps\left(\frac{\rho}{1-\rho}+\frac{1-\rho}{4(1+\rho)}\right),
\]
the non-expansiveness of $\mathcal I - \mathcal S_{\tilde \alpha}$ ensures that
\[
\|u(h)-\mathcal S_{\tilde \alpha}u(h)\|\geq \eps\left(\frac{1}{1-\rho}-\frac{\rho}{1-\rho}-\frac{1-\rho}{4(1+\rho)}\right)\geq \frac{3}{4}\,\frac{1-\rho}{1+\rho}\,\eps\geq 3\|u(h)-\mathcal S_{\bar \alpha}u(h)\|.
\]
This implies that $\tilde \alpha\geq 3\bar \alpha$ and thus
\[
\begin{aligned}
\rank(\mathcal S_{\tilde \alpha}\Phi \mathbf u^K(h)) &\leq \rank(\mathcal S_{3\bar \alpha}\Phi \mathbf u^K(h))  \\ &\leq \rank(\mathcal S_{\bar \alpha}u(h))+\frac{\|\Phi \mathbf u^K(h)-u(h)\|^2}{4\bar \alpha^2}\\
&\leq \frac{\| u - \mathcal S_{\bar \alpha} u \|_{\mathcal W(0,h)}^2}{\bar \alpha^2}+\frac{\eps^2}{4\bar \alpha^2(1-\rho)^2}
\leq \frac{\bar r\rho}{2(1-\rho)^2}.
\end{aligned}
\]
Now, consider the recompressed version $\tilde u(h)=\mathcal R_\delta(\Phi \mathbf u^K(h))$. If $\sigma_{\min}(\tilde u(h))>\tilde \alpha$ (which could happen even when $\delta<\eps$, since hard thresholding makes a smaller error than soft thresholding), then
\[
\rank(\tilde u(h))\leq \rank(\mathcal S_{\tilde \alpha}\Phi \mathbf u^K)\leq \frac{\bar r\rho}{2(1-\rho)^2}.
\]
Otherwise, for $v=\Phi \mathbf u^K(h)-\mathcal S_{\tilde \alpha}\Phi \mathbf u^K(h)$, it holds that
\[
\rank(\tilde u(h)) =\rank(\mathcal R_\delta(\Phi \mathbf u^K(h))) = \rank(\mathcal R_\delta(v)),
\]
and we conclude by applying Lemma~\ref{lemmarecompressionranks} to $v$:
\[
\rank(\mathcal R_\delta(v)) \leq 1+\frac{\eps^2/(1-\rho)^2-\delta^2}{\eps^2/(1-\rho)^2}\rank(v)=1+\frac{\eps^2-\delta^2(1-\rho)^2}{\eps^2}\rank(\Phi \mathbf u^K(h)).\qedhere
\]
\end{proof}

We finish this section by combining the previous lemmas.

\begin{prop}
\label{prop_local_ranks}
Assume that $\rho\leq 1/2$, and let
\[
\eps\geq 4 \, \frac{1+\rho}{1-\rho} \, \| u - u^*\|_{\mathcal W(0,h)} \quad \text{and} \quad \delta = \frac{\rho \eps}{1-\rho}.
\]
If the initial rank can be bounded by
\begin{equation*}
	\rank(u_0^*)\leq \frac{1+J\rank(V)\bar r}{\rho^2},
\end{equation*}
then the rank of the approximation at the right subinterval boundary can be estimated by
\[
\rank(\tilde u(h))\leq  \frac{1+J\rank(V)\bar r}{\rho^2},
\]
and all other intermediate ranks, that is, the ranks of the iterates $\mathbf u^k$ and $\mathbf \Phi \mathbf u^k$ produced by the algorithm for $0\leq k\leq K$, are bounded by $\frac{1+\rho^2}{\rho^2}(1+J\rank(V)\bar r)$.
\end{prop}
\begin{proof}
Combining Lemmas~\ref{lemmarankofutilde}, \ref{lemmarankofPhiuk} and \ref{lemmaboundingtherankofuk}, we either have $\rank(\tilde u(h))\leq \frac{\bar r\rho}{2(1-\rho)^2}\leq \bar r$, or 
\begin{align*}
\rank(\tilde u(h))
&\leq 1+\frac{\eps^2-(1-\rho)^2\delta^2}{\eps^2}\biggl(\frac{1+J\rank(V)\bar r}{\rho^2}+J\rank(V)\bar r\biggr)\\
&\leq 1+\frac{1-\rho^2}{\rho^2}+(1-\rho^2)J\rank(V)\bar r\biggl(\frac{1}{\rho^2}+1\biggr)\\
&= \frac{1}{\rho^2}+J\rank(V)\bar r\biggl(\frac{1}{\rho^2}-\rho^2\biggr).
\end{align*}
The bound on intermediate ranks follows by another application of Lemmas~\ref{lemmarankofPhiuk} and \ref{lemmaboundingtherankofuk}.
\end{proof}

\begin{remark}
The assumption $\rho\leq 1/2$ is easily satisfied, by taking a time step $h$ twice smaller than what is already needed for the contractivity in Proposition~\ref{prop_contractivity}. It is included only to avoid further technicalities, by allowing to bound
the first option in the maximum of Lemma~\ref{lemmarankofutilde} by $\bar r$.
If it was removed, the bounds on the ranks in Proposition~\ref{prop_local_ranks} would have to be replaced by
\[
\max\biggl\{ \frac{\bar r\rho}{2(1-\rho)^2},\frac{1+J\rank(V)\bar r}{\rho^2}\biggr \} .
\]
\end{remark}

\subsection{Global bounds}
We have now all necessary tools to prove the main theorem. Assume that we know some upper bounds on the constants $\kappa_J$ and $\kappa_{2J}$, and denote by $\tilde u_n$ the approximation generated by the scheme at time $nh$, $1 \leq n \leq N = T/h$.
Let $\eta$ denote the target accuracy, then our target for the soft threshold is
\[
\alpha_n=\inf_{(n-1)h\leq t\leq nh}\min_{\|u(t)-v\|\leq \eta} \sigma_{\max}(v)=\max\bigl\{\alpha>0 \colon \forall t\in [(n-1)h,nh],\;\|u(t) - \mathcal S_\alpha u(t)\|\leq \eta \bigr\},
\]
since the minimizer is of the form $v=\mathcal S_{\alpha(t)} u(t)$ for some $\alpha(t)>0$. Our objective for the ranks is then given by
\[
r_n=\frac{\eta^2}{\alpha_n^2}=\frac{\|u-\mathcal S_{\alpha_n}u\|_{\mathcal W((n-1)h,nh)}^2}{\alpha_n^2}.
\]
We further define
\[
\xi_n= \bigl(\eta+\kappa_Jh^{J+1}+n\kappa_{2J}h^{2J+1}\bigr)\exp\left(8\rho n\frac{1+\rho}{(1-\rho)^3}\right)
\]
and make the choice
\[
\eps_n=4\xi_{n-1}\frac{1+\rho}{(1-\rho)^2}\quad\text{and}\quad  \delta_n=\frac{\rho \eps_n}{1-\rho}.
\]

\begin{theorem}
\label{thm:main}
The proposed scheme achieves a global error bounded by
\[
\sup_{1 \leq n \leq N}\|\tilde u_n - u(nh)\|\leq \left(\eta + \kappa_J\,h^{J+1}+\kappa_{2J}\,h^{2J}\,T\right)\exp\left(8\Lambda_JC_V\frac{1+\rho}{(1-\rho)^3}T\right).
\]
The ranks of the scheme are bounded by
\[
\rank(\tilde u_n)\leq  \frac{16}{\rho^3}\left(\frac{1+\rho}{1-\rho}\right)^2\,\Big(1+J\rank(V)\max\{r_1, r_2, \dots,r_n\}\Big),
\]
for all $1 \leq n \leq N$, and the intermediate ranks are at most twice as large.
\end{theorem}

\begin{proof}
For the global error bound, we prove by induction that
\begin{equation}
\label{inductionwithxin}
\|\tilde u_n-u(nh)\|\leq \xi_n-\kappa_Jh^{J+1}.
\end{equation}
This is sufficient since $\rho=\Lambda_JC_Vh$ and $nh\leq T$.
We initialise the scheme with $\tilde u_0 = \mathcal H_{\alpha_0} u(0)$,
where $\alpha_0=\min_{\|u(0)-v\|\leq \eta}\sigma_{\max}(v)\geq \alpha_1$ and $r_0=\eta^2/\alpha_0^2\leq r_1$, which ensures that
\[
\|\tilde u_0-u(0)\|\leq \eta = \xi_0 - \kappa_Jh^{J+1}
\quad\text{and}\quad
\rank(\tilde u_0)\leq \frac{\|u(0)-\mathcal S_{\alpha_0}u(0)\|^2}{\alpha_0^2}= r_0.
\]
For $1\leq n \leq N$, assume that \eqref{inductionwithxin} holds at index $n-1$, and denote $u^\diamond_0=u((n-1)h)$ and $u^*_0=\tilde u_{n-1}$ the values of the solution and the scheme at time $(n-1)h$. By Lemma~\ref{localerrorboundinside},
\[
\|u^*-u\|_{\cW((n-1)h,nh)} \leq 
\frac{\kappa_J h^{J+1} + \|\tilde u_{n-1}-u((n-1)h)\|}{1-\rho}\leq \frac{\xi_{n-1}}{1-\rho}=\frac{1-\rho}{1+\rho}\frac{\eps_n}4
\]
hence we satisfy the condition of Lemma~\ref{lemmaboundingtherankofuk}. In addition, by Proposition~\ref{prop_local_error},
\begin{align*}
\|\tilde u_n-u(nh)\|
&\leq \delta_n+\frac{\rho\eps_n}{1-\rho}+\kappa_{2J}h^{2J+1}+\xi_{n-1}-\kappa_Jh^{J+1}\\
&\leq \biggl(1+8\rho\frac{1+\rho}{(1-\rho)^3}\biggr)\xi_{n-1}+\kappa_{2J}h^{2J+1}-\kappa_Jh^{J+1}\leq \xi_n-\kappa_Jh^{J+1},
\end{align*}
where we used the fact that $1 + x \leq e^x$, and this completes the induction.

Next, we define $\bar{\alpha}_n$ for each interval $[(n-1)h,nh]$ such that
\begin{equation*}
	\| u - \mathcal S_{\bar \alpha_n} u \|_{\mathcal W((n-1)h,nh)} = \frac{1-\rho}{1 + \rho}\, \frac{\eps_n}{4},
\end{equation*}
and again set $\bar{r}_n = \varepsilon_n^2 / (\rho \bar{\alpha}_n^2)$. For the ranks, observe that for all $n\leq T/h$, we have $\frac{1-\rho}{1+\rho}\frac{\eps_n}4\geq \eta$, thus $\bar \alpha_n \geq \alpha_n$ and therefore
\[
\frac{\|u(t)-\mathcal S_{\bar \alpha_n}u(t)\|^2}{\bar \alpha_n^2}\leq \frac{\|u(t)-\mathcal S_{\alpha_n}u(t)\|^2}{\alpha_n^2} \leq r_n.
\]
In particular, for any $1 \leq n \leq N$, the value of $\bar r_n$ satisfies
\[
\bar r_n \leq \frac{16}{\rho}\frac{(1+\rho)^2}{(1-\rho)^2} r_n,
\]
and we conclude by induction on Proposition~\ref{prop_local_ranks}.
\end{proof}

\begin{remark}
We actually defined a scheme on the whole interval $[0,T]$, which satisfies
\[
\sup_{t\in [0,T]}\|\tilde u(t)-u(t)\|\leq \frac{1}{1-\rho}\left(\eta + \kappa_J\,h^{J+1}+\kappa_{2J}\,h^{2J}\,T\right)\exp\left(8\Lambda_JC_V\frac{1+\rho}{(1-\rho)^3}T\right).
\]
The proof is essentially the same, with one additional use of Lemma~\ref{localerrorboundinside} in the interval containing $t$, combined with \eqref{inductionwithxin} at $n=\lfloor t/h\rfloor$.
The ranks of $\tilde u(t)$ are bounded in the same way as above, which justifies the quasi-optimality statements made in the introduction.
\end{remark}

\begin{remark}
In Theorem~\ref{thm:main}, the bound on the ranks can only grow as $n$ grows.
If the ranks of the solution decrease over time,
a closer inspection at the proof of Proposition~\ref{prop_local_ranks} reveals that the bound on the ranks of the scheme can actually decrease, by a factor at best $1-\rho^2$, from one interval to the next. The rate of decay $(1-\rho^2)^n$ is however quite pessimistic, and the ranks computed numerically can track a faster decay, 
as shown in Section \ref{sec:numexp} (see Figure~\ref{fig:parabolic}) for the parabolic case.
\end{remark}

\begin{remark}[Analysis of the SDC scheme]
	A central difference between Algorithms~\ref{alg:Picard} and \ref{alg:SDC} is the fact that the soft thresholding operator is applied to the intermediate approximations at different points in the algorithm. While for the Gauss-Legendre-Picard iteration this is done after each application of the fixed-point iteration $\mathbf{\Phi}_{\mathrm{Picard}}$ in matrix form, the same can not be said about the SDC iterations. Here, the soft thresholding operator is applied after the component-wise application of $\mathbf{\Phi}_{\mathrm{SDC}}$, making it more difficult to use the non-expansiveness property. One possibility to fully analyze Algorithm \ref{alg:SDC} would thus be to decouple $\mathbf{\Phi}_{\mathrm{SDC}}$ into fixed-point iterations on each Gauss-Legendre node. Combined with the general contractivity of the SDC scheme, one could think about using different soft thresholding parameters (and strategies to decrease them) for the $J$ decoupled fixed-point iterations.
\end{remark}

\section{Adaptation to parabolic problems}\label{sec:parabolic}

An essential ingredient of the described iterative algorithms is certainly the choice of quadrature points $t_1, t_2, \dots, t_J$. While for the time-dependent Schrödinger equation, we had numerous reasons to take Gauss-Legendre nodes, they turn out to be rather unfavorable in the parabolic setting: because of the 
isometry-preservation,
the damping of high frequencies is correctly reproduced only for very small time steps $h$.

Recall the scalar test problem \eqref{eq:exptestproblem} used to analyze the stability properties of Runge-Kutta schemes. A somewhat stronger notion of stability can be defined by the so-called \textit{L-stable} methods, for which the stability function $R$ needs to additionally satisfy
\begin{equation*}
|R(z)| \rightarrow 0 \ \text{ as } \ \text{Re}(z) \rightarrow -\infty,
\end{equation*}
that is, the numerical method reproduces the damping of modes with negative singular values. Note that this rules out preservation of isometries, since for rational functions $R$ we have $\lim_{x\to -\infty}|R(x)|=\lim_{y\to \infty}|R(iy)|$. Such a method can be obtained by instead using Radau-Legendre quadrature nodes on the respective subintervals. 

 \begin{example}[Radau-Legendre quadrature]
 		In contrast to the Gauss-Legendre quadrature, we force the last node to coincide with the right interval boundary $t_J = t_0+h$.
Optimizing over the remaining degrees of freedom, the quadrature formula \eqref{quadrature_approx} can at most attain a degree of accuracy $2J-2$, that is, one less than the Gauss-Legendre quadrature formula.
The optimal choice of points can again be computed by considering a tridiagonal matrix: it turns out it suffices to modify the bottom right entry of matrix \eqref{tridiaggauss}, in such a way that the largest eigenvalue of the resulting matrix is exactly $t_0+h$, see for example \cite{alqahtani_computation_2025}. The Radau-Legendre quadrature weights and nodes can then be analogously derived from the eigenvalues and eigenvectors of the modified matrix. The resulting Gauss method  \eqref{butcher} has then the order of convergence $2J-1$.
 \end{example}
 
 The modified fixed-point formulation for the parabolic problem \eqref{eq_parabolic} was discussed in Example \ref{ex:mod_pf_parabolic}. Unfortunately, as already noted, we cannot directly translate the application of the Radau-Legendre-Picard iteration to this formulation,
 because the equivalent of formula \eqref{eq_untwisted_phi} for the definition of $\mathbf \Phi$ would include exponential operators of the form
 $e^{a(t_j - t_m) \Delta}$, which are unbounded when $t_j - t_m < 0$ and would thus lead to an extremely unfavorable coupling to the resolution of the spatial discretization. As before, we aim to avoid such a coupling between time steps size and spatial discretization in order to allow larger time steps.
	
In order to avoid taking exponentials of negative Laplacians, we instead only interpolate the second part of the integrand in \eqref{parabolic_duhamel}, namely $\mathcal G u(s) + f(s)$, by a Lagrange polynomial
\[
p(s)= \sum_{m=1}^J \bigl( \mathcal G u(t_m) + f(t_m) \bigr) \ell_m(s)=\sum_{m=1}^J 2b\partial_{x_1}\partial_{x_2} u(t_m) \ell_m(s) + f.
\]
To compute the fixed-point application $\mathbf \Phi_{\mathrm{Picard}}(\mathbf u)_j$, for $1\leq j\leq J$, we thus need to approximate integrals of the type
\begin{equation*}
	\int_{t_{j-1}}^{t_j} e^{a(t_j-s) \Delta} p(s) \, ds.
\end{equation*}
As the above integrand is still singular at $s=t_j$, we divide the integration domain by applying $N_j\in \N$ bisections that concentrate near $t_j$, that is,
\[
[t_{j-1}, t_j]=\left[t_j-\tau_j,t_j-\frac{\tau_j}{2}\right]\cup\dots\cup \left[t_j-\frac{\tau_j}{2^{N_j-1}},t_j-\frac{\tau_j}{2^{N_j}}\right]\cup \left[t_j-\frac{\tau_j}{2^{N_j}}, t_j\right],
\]
where $\tau_j=t_j-t_{j-1}$, and use a quadrature formula $\mathcal Q_a^b(g)\approx \int_a^b g(s)\,ds$ of moderate degree on each of these intervals.
In the end, Duhamel's formula \eqref{parabolic_duhamel} is approximated by
\[
\mathbf \Phi_{\mathrm{Picard}}(\mathbf u)_j=e^{a\tau_j\Delta}\mathbf \Phi_{\mathrm{Picard}}(\mathbf u)_{j-1}+\sum_{n=1}^{N_j} \mathcal Q_{t_j-\frac{\tau_j}{2^{n-1}}}^{t_j-\frac{\tau_j}{2^n}}\big(e^{a(t_j-\cdot) \Delta} p(\cdot)\big)+\mathcal Q_{t_j-\frac{\tau_j}{2^{N_j}}}^{t_j}\big(e^{a(t_j-\cdot) \Delta} p(\cdot)\big).
\]
Note that we only need to compute the fixed-point application at the nodes $t_j$, since the interval boundary $t_0+h$ is equal to $t_J$.

Concerning $\mathbf \Phi_{\mathrm{SDC}}$, we use the same construction, with an additional explicit low-order integrator $\Psi_{t_{j-1},t_j}$ as in the dispersive setting \eqref{eq:SDC}. In the numerical tests of Section~\ref{sec:numexp}, we take the explicit Euler integrator
\[
\Psi_{t_{j-1},t_j}=\tau_j e^{a\tau_j\Delta}\mathcal G
\]
applied to $\mathbf \Phi_{\mathrm{SDC}}(\mathbf u)_{j-1}-u_{j-1}$, and do not need to consider the source term $f$ since the correction is a difference of two solutions.

Due to the introduced secondary quadrature, the resulting (modified) Radau-Legendre-Picard iteration is no longer equivalent to the collocation formulation and corresponding Runge-Kutta method \eqref{butcher}. However, as the numerical tests in the following section show, the resulting schemes still perform as expected, which makes them promising candidates for further analysis.

\section{Numerical experiments}\label{sec:numexp}

In this section, we present numerical experiments for the time-dependent Schr\"odinger and parabolic model problems \eqref{eq_Schrodinger} and \eqref{eq_parabolic}. All tests were performed using Julia version 1.10.0. 
Recall that we are considering evolution problems on product domains $\Omega_1 \times \Omega_2$, of the form
\begin{equation*}
	\partial_t{u}(t) = \mathcal{A}_t\bigl(u(t)\bigr),\qquad  t\in [0,T].
\end{equation*}
Here, we choose $\Omega_1 = \Omega_2 = [0,1]$ and impose homogeneous Dirichlet boundary conditions. We thus have the $L^2(\Omega)$ orthonormal basis expansion of the solution 
\begin{equation*}
	u(t)(x_1,x_2) = \sum_{k_1, k_2 = 1}^{\infty} 2 u_{k_1,k_2}(t) \sin(\pi k_1 x_1)\sin(\pi k_2 x_2),
\end{equation*}
with $u_{k_1,k_2}(t) \in \mathbb{C}$ (or $\mathbb R$ in the parabolic setting) for $k_1, k_2 \in \mathbb N$. 
Galerkin projection onto the first $K$ frequencies in each coordinate leads to a system of ordinary differential equations
\begin{equation*}
	\partial_t{\mathbf{u}}(t) = \mathbf{A}_t \bigl(\mathbf{u}(t)\bigr),\qquad  t\in [0,T],
\end{equation*}
for some discrete operator $\mathbf A_t: \mathbb C^{K\times K}\to  \mathbb C^{K\times K}$,
and $\mathbf{u}(t)[k_1,k_2] \approx u_{k_1,k_2}(t)$. Note that using this Fourier basis expansion, we can compute the involved matrix exponentials explicitly: the Laplacian operator in matrix form is diagonal and thus its exponential can be computed by applying the exponentials to its eigenvalues. Our aim is to obtain low-rank approximations of the discrete solution $\mathbf{u}(t)$.

\subsection{Schrödinger equation}

First, we consider the time-dependent Schrödinger equation with potential
\begin{equation*}
	V(x_1,x_2) = \cos(n \pi x_1)\cos(m \pi x_2)
\end{equation*}
with some $m,n \in \N$. The spatial semidiscretization yields
\begin{equation*}
	\partial_t{\mathbf{u}}(t) = -i \left( \mathbf{S} \otimes \mathbf{I} + \mathbf{I} \otimes \mathbf{S} + \mathbf{V}_n \otimes \mathbf{V}_m \right) \mathbf{u}(t),
\end{equation*}
with the matrices
\begin{equation*}
\begin{aligned}
	\mathbf{S}[k_1,k_2] &= \int_0^1 2 \left(\partial_{x_1} \sin(k_1 \pi x_1)\right) \left(\partial_{x_1} \sin(k_2 \pi x_1)\right) \, dx_1 , \\
	\mathbf{V}_n[k_1,k_2] & = \int_0^1 2 \cos(n \pi x_1) \sin(k_1 \pi x_1) \sin(k_2 \pi x_1) \, dx_1.
\end{aligned}
\end{equation*}
As initial data, we choose
\begin{equation*}
	\mathbf u _0 = \mathbf{v}_K \otimes \mathbf v_K, \quad \text{where} \quad \mathbf v_K \in \C^{K} \quad \text{with} \quad v_K[k] = \frac{1}{k}
\end{equation*}
and $K = 300$ frequencies. In the following, we use $n = m = 1$ for the cosine potential and a time step $h=0.1$ with $N=5$ intervals and $J=11$ Gauss-Legendre quadrature nodes. Note that here it is important that with the fixed-point formulation and iterative schemes that we use, $h$ can be chosen independently of the spatial discretization determined by $K$.

For the implementations of the two iterative solvers, we follow Algorithms \ref{alg:Picard} and \ref{alg:SDC}. For target accuracies of $\eta=10^{-3}$ and $\eta=10^{-6}$, the recompression parameter at subintervals boundaries is taken as $\delta_n=10^{-4}$ and $\delta_n= 10^{-7}$, respectively.
In practice, when computing the updated approximations $\mathbf v^{k+1}$, we use additional recompressions up to a relative error of $10^{-3}$ after every operation on low-rank functions.  Note that in case of SDC, we follow the approach described in \cite{DGR00} for the computation of the $\phi_j^k$ in line 7 of the algorithm from the error and residual functions and hence need here an additional smaller recompression tolerance of $10^{-6}$ (relative to the current approximation of the residual function), as already mendtioned in Section 3.2, due to an observed accumulation of recompression errors.

\begin{figure}[t]
	\centering
	\begin{minipage}{0.5\textwidth}
		\centering
		\includegraphics[width=\textwidth]{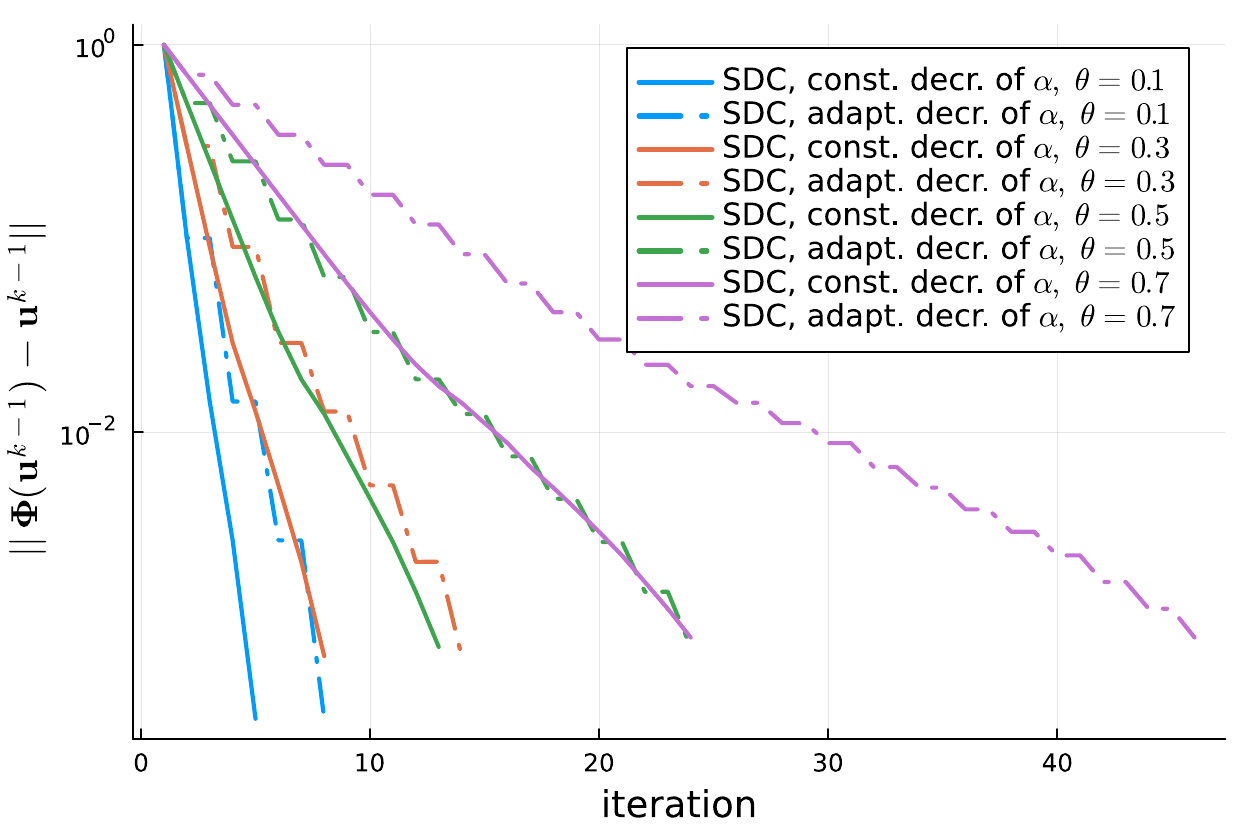}
	\end{minipage}%
	\begin{minipage}{0.5\textwidth}
		\centering
		\includegraphics[width=\textwidth]{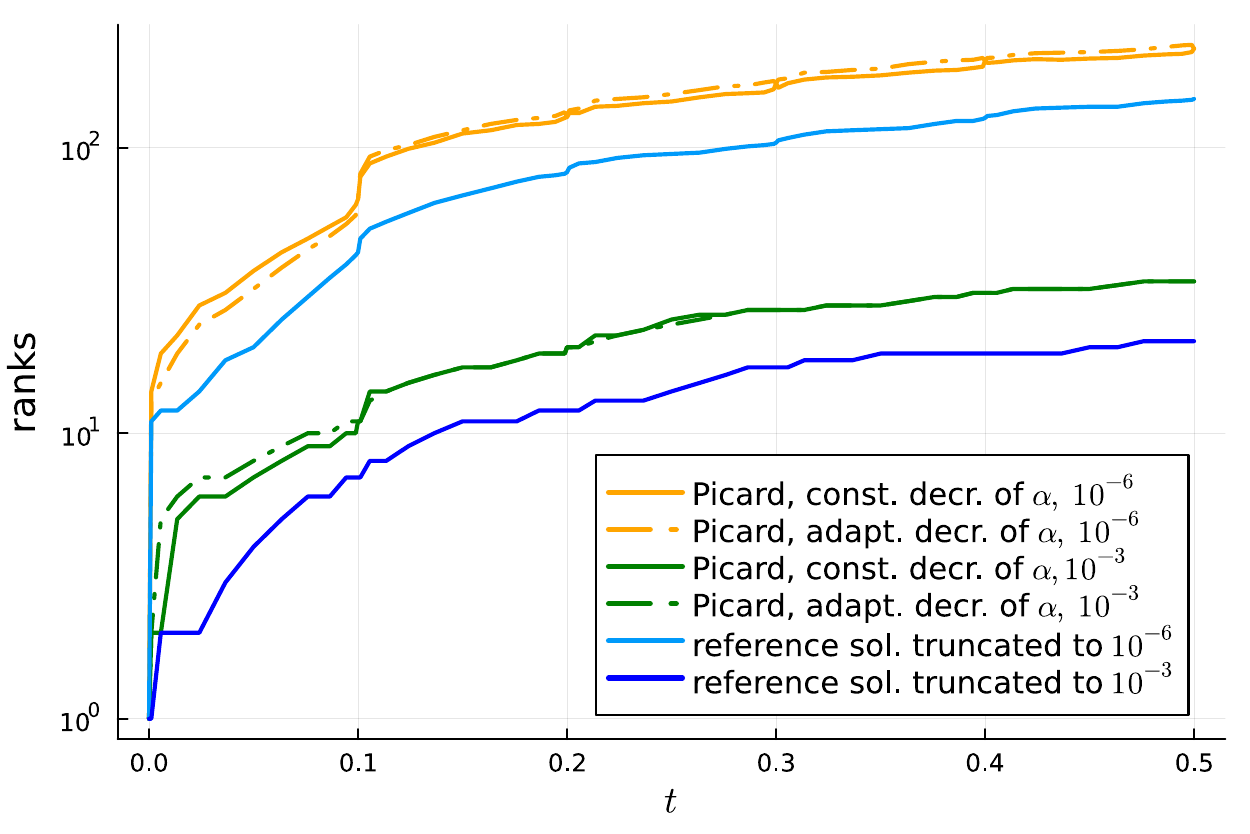}
	\end{minipage}
	\caption{Left: decrease of the residual error estimator over the iterations for varying values of $\theta$ on a single interval $[t_0,t_0+h]$. Right: evolution of the ranks using different strategies for the decrease of $\alpha$ and different iteration tolerances ($10^{-3}$ and $10^{-6}$) compared to the corresponding optimal ranks.}
	\label{fig:schroedinger_ranks}
\end{figure}

Another important parameter is the rate by which $\alpha$ is decreased. As can be deduced from the left part of Figure \ref{fig:schroedinger_ranks}, smaller values of $\theta \in (0,1)$ allow for a faster convergence of the fixed-point iteration,
 even though we only have rank bounds for $\theta\geq\rho$, and they are tighter for larger values of $\theta$.
 One can also observe that the adaptive decrease of $\alpha$ requires about twice as many iterations as the constant decrease of $\alpha$. This can be explained by the following consideration: in the adaptive setting, $\alpha_{k+1}$ is smaller than $\alpha_k$ whenever
\begin{equation*}
	\bigl\| \mathbf u^{k+1} - \mathbf u^k \bigr\|_J \leq \frac{1}{6} \bigl\| \mathbf u^k - \mathbf{\Phi} \mathbf u^k \bigr\|_J,
\end{equation*}
is satisfied, where $\mathbf{\Phi}$ now denotes the discretized fixed-point application on $\bigl(\C^{K\times K} \bigr)^J$.
This condition only seems to occur when $\mathbf u^{k+1}$ and $\mathbf u^k$ have been thresholded with the same parameter $\alpha_k=\alpha_{k-1}$, because in that case they are both close to the modified fixed point $\mathbf u_{\alpha_k}=\mathbf u_{\alpha_{k-1}}$.
Hence, $\alpha$ can only be decreased every second iteration.
Moreover, the distance to the fixed point is dominated by the soft thresholding error, which is almost proportional to $\alpha$, that is,
\begin{equation}
\label{error_dominated_by_alpha}
\|\mathbf u^{k+1} -\mathbf u^*\|_J \approx
\| (\mathcal I-\mathcal S_{\alpha_{k}}) \bigl(\mathbf  \Phi \mathbf u^{k} \bigr) \|_J \approx \alpha_{k} \sqrt{\mathrm{rank}( \mathbf u^{k+1} )},
\end{equation}
justifying the dependence of the accuracy in the number of executed iterations.
In the following, we thus use $\theta=\frac12$ and $\theta=\frac15$ for the constant and adaptive decrease of $\alpha$, respectively, in order to maintain comparable numbers of iterations. Note that the illustration on the left of Figure \ref{fig:schroedinger_ranks} does not change visibly if we replace the SDC scheme by the simpler Picard iteration, which is another indicator of the fact that the iteration error is dominated by the soft thresholding error.

We continue our numerical investigations by considering the performance of the Picard iterative solver combined with different iteration tolerances. The obtained evolution of ranks is visualized on the right of Figure \ref{fig:schroedinger_ranks}: for both iteration tolerances the Picard method combined with the two possible reduction strategies of $\alpha$ follows, up to a constant, the trajectories of the optimal ranks.
The fixed-point reference solution was computed in full matrix format using the Picard iteration with iteration tolerance $10^{-12}$ and without any recompression or truncation.

\begin{figure}[t]
	\centering
	\begin{minipage}{0.5\textwidth}
		\centering
		\includegraphics[width=\textwidth]{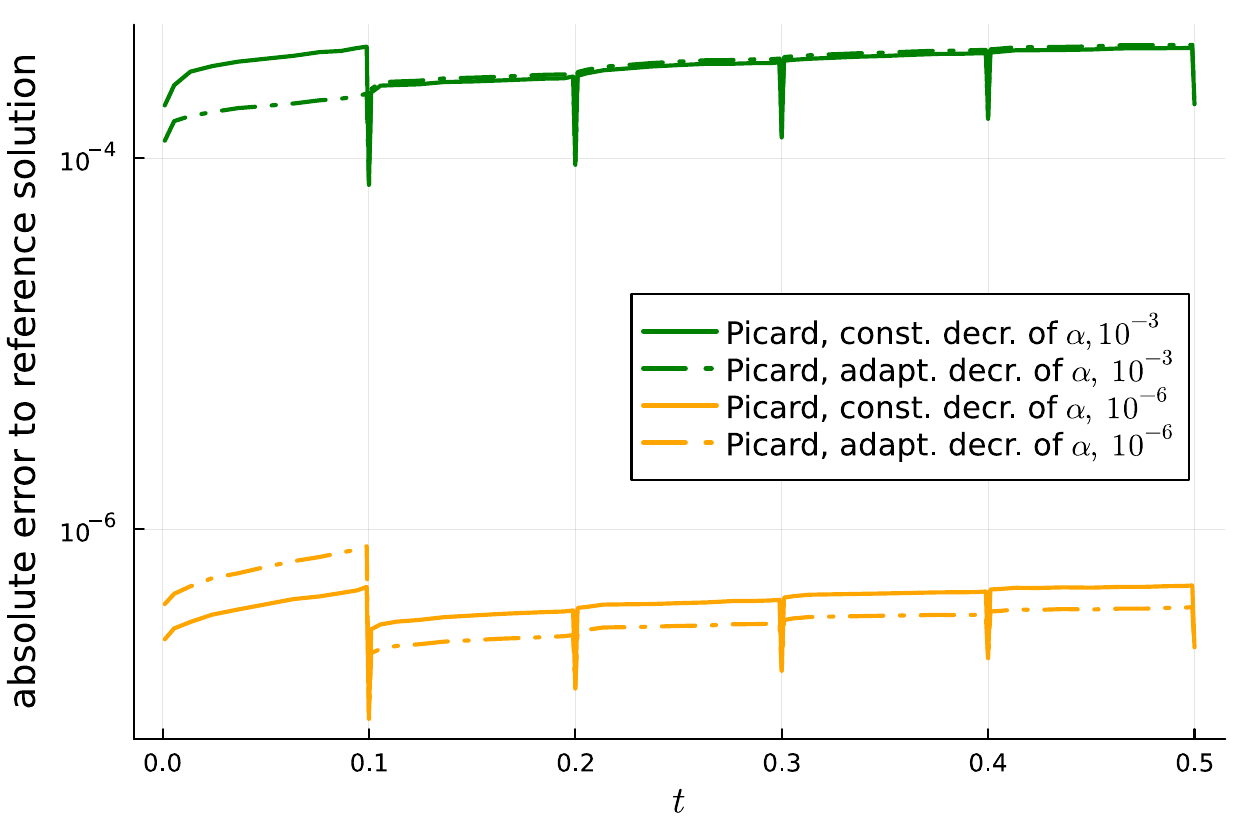}
	\end{minipage}%
	\begin{minipage}{0.5\textwidth}
		\centering
		\includegraphics[width=\textwidth]{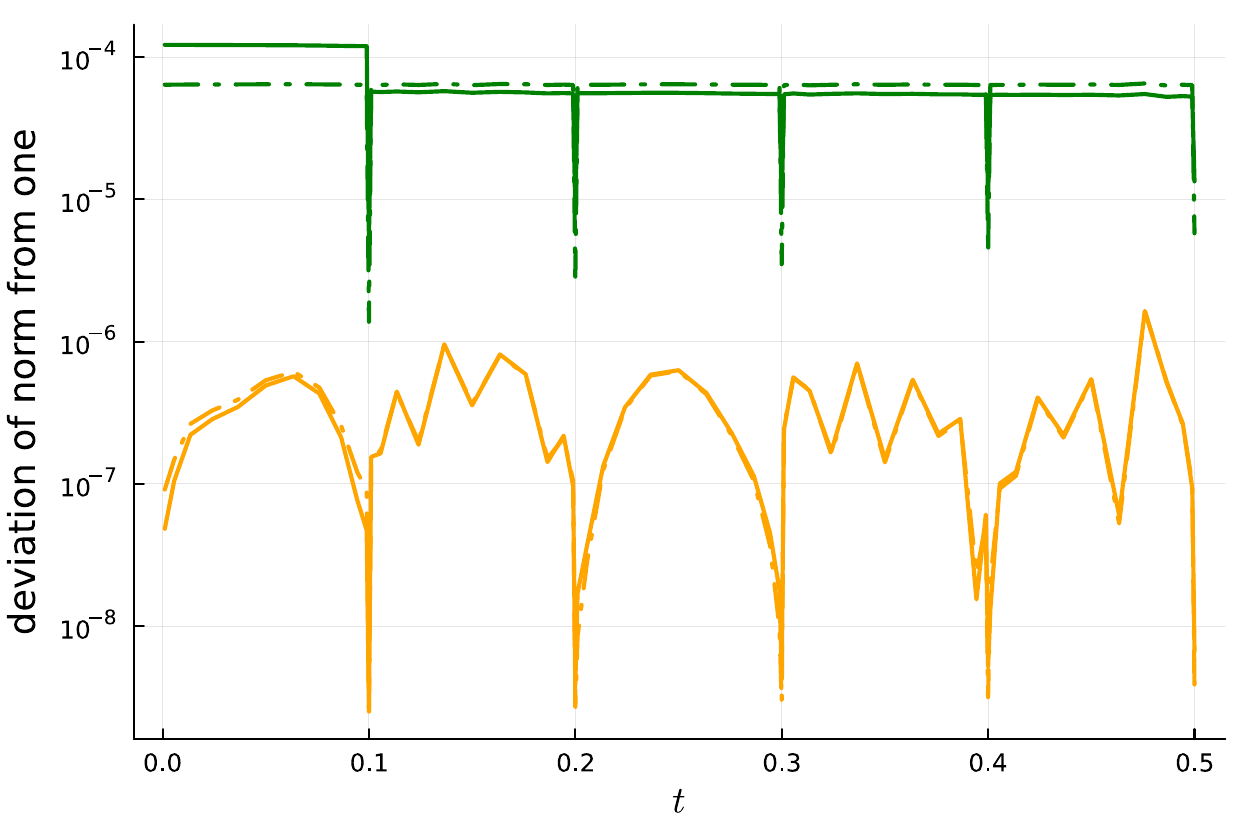}
	\end{minipage}
	\caption{Left: evolution of the absolute error with respect to a reference solution. Right: deviation of the norm of the computed approximations from $\|u_0\|=1$ over time.
	}
	\label{fig:schroedinger_norms}
\end{figure}

As can be deduced from the left part of Figure \ref{fig:schroedinger_norms}, the absolute error accumulates only slowly over time and remains below the tolerances used for the iterative solvers over the entire interval $[0,0.5]$. In particular, we have a significantly smaller error at the subinterval boundaries: while on the inner nodes, the error is in general expected to be of order $\mathcal{O}(h^{J+1})$,
we expect an accuracy of $\mathcal{O}(h^{2J})$ at the subinterval boundaries.

In addition, the recompression at the boundary nodes is expected to preserve the $L^2(\Omega)$ norm of the scheme much better than soft thresholding.
Indeed, for a function $u$ with singular values $(\sigma_k)$ and a threshold $\alpha\ll 1$, 
\begin{equation*}
	\| u \|^2 - \| \mathcal S_{\alpha}(u) \|^2 = \sum_{k\in \N} \sigma_k^2-\max(\sigma_k-\alpha,0)^2
	=\sum_{\sigma_k\geq \alpha} (2\alpha\sigma_k-\alpha^2)+\sum_{\sigma_k<\alpha}\sigma_k^2
	=\mathcal{O}(\alpha).
\end{equation*} 
On the other hand, the recompression behaves similarly to hard thresholding, for which the norm deviation is much smaller
\begin{equation*}
	\| u \|^2 - \| \mathcal H_{\alpha}(u) \|^2 = \sum_{\sigma_k>\alpha} \sigma_k^2=\mathcal{O}(\alpha^2).
\end{equation*}
This can be seen on the right of Figure \ref{fig:schroedinger_norms}. In fact, we observe a preservation of the norm at the subinterval boundaries up to iteration and recompression tolerances.

\subsection{Parabolic problem}
The second running example that we consider is the heat equation with anisotropic diffusion \eqref{eq_parabolic}. The problem is approximated, analogously to the first example, by an ODE of the form
\begin{equation*}
	\partial_t{ \mathbf{u} } = \left( - a(\mathbf{S} \otimes \mathbf{I} + \mathbf{I} \otimes \mathbf{S}) + 2b \mathbf{B} \otimes \mathbf{B} \right)\mathbf{u} + \mathbf{f} \otimes \mathbf{f}, \qquad \mathbf{u} \in \mathbb R ^{K \times K},
\end{equation*}
with parameters $K=500$, $a = 1$ and $b = -\frac{1}{2}$, and where the matrix $\mathbf{B}$ is given by
\begin{equation*}
	\mathbf{B}[k_1,k_2] = \int_0^1 2 (\partial_{x_1} \sin(k_1 \pi x_1)) \sin(k_2 \pi x_1) \, dx_1.
\end{equation*}
The initial data $\mathbf u_0$ and source term $\mathbf f$ are chosen in order to first observe an increase in ranks, followed by a quick decrease due to the smoothing properties of the equation, and finally a slight long-term rise of the ranks caused by the source term. More precisely, we set
\begin{equation*}
	\mathbf u_0[k_1,k_2] = \begin{cases}
		\delta_{k_1,k_2}, \quad &\text{for } 21 \leq k_1,k_2 \leq 30, \\
		0, \quad &\text{otherwise},
	\end{cases} \quad \text{and} \quad \mathbf f[k] = \begin{cases}
	k, \quad & \text{for } 1 \leq k \leq 10, \\
	0, \quad & \text{otherwise}.
	\end{cases}
\end{equation*}
In order to best capture the initial rank evolution, we use a stepsize of $h=0.001$ with $N=10$ intervals and $J=11$ Radau-Legendre quadrature nodes. Note that with this choice of points, since $t_J = t_{0}+h$, we also apply the soft thresholding operator to the approximations computed at the right subinterval boundaries instead of the recompression operator. Larger step sizes would be possible for this problem without any stability issues. However, resolving the initial rank fluctuations then requires a correspondingly higher polynomial degree.

For the secondary quadrature, we take $N_j=5$ refinements by bisection of the intervals $[t_0,t_j]$ for $1 \leq j \leq J$ and use Radau-Legendre quadrature $\mathcal Q_a^b$ with 5 nodes on each of the resulting subintervals.
Without this special form of integral approximation, a much larger number of interpolation points (at least $7J$) is needed to reach tolerances below an accuracy of $10^{-9}$. Otherwise, the fixed-point iterations only converge up to a certain error, and then rapidly diverge.

We again use Algorithms \ref{alg:Picard} and \ref{alg:SDC}, respectively, and make use of additional low-rank recompressions in intermediate steps. For recompressing intermediate solutions after the summation of two low-rank matrices, we use a relative tolerance of $10^{-4}$, see also the previous example, except when we approximate the residual function within the SDC iteration, where a lower relative recompression tolerance of $10^{-6}$ is required.

The fixed-point reference solution was computed in full matrix format using the Picard iteration with an iteration tolerance of $10^{-9}$, and 10 dyadic refinements with 5 quadrature nodes on each subinterval for the secondary quadrature.

\begin{figure}[t]
	\centering
	\begin{minipage}{0.5\textwidth}
		\centering
		\includegraphics[width=\textwidth]{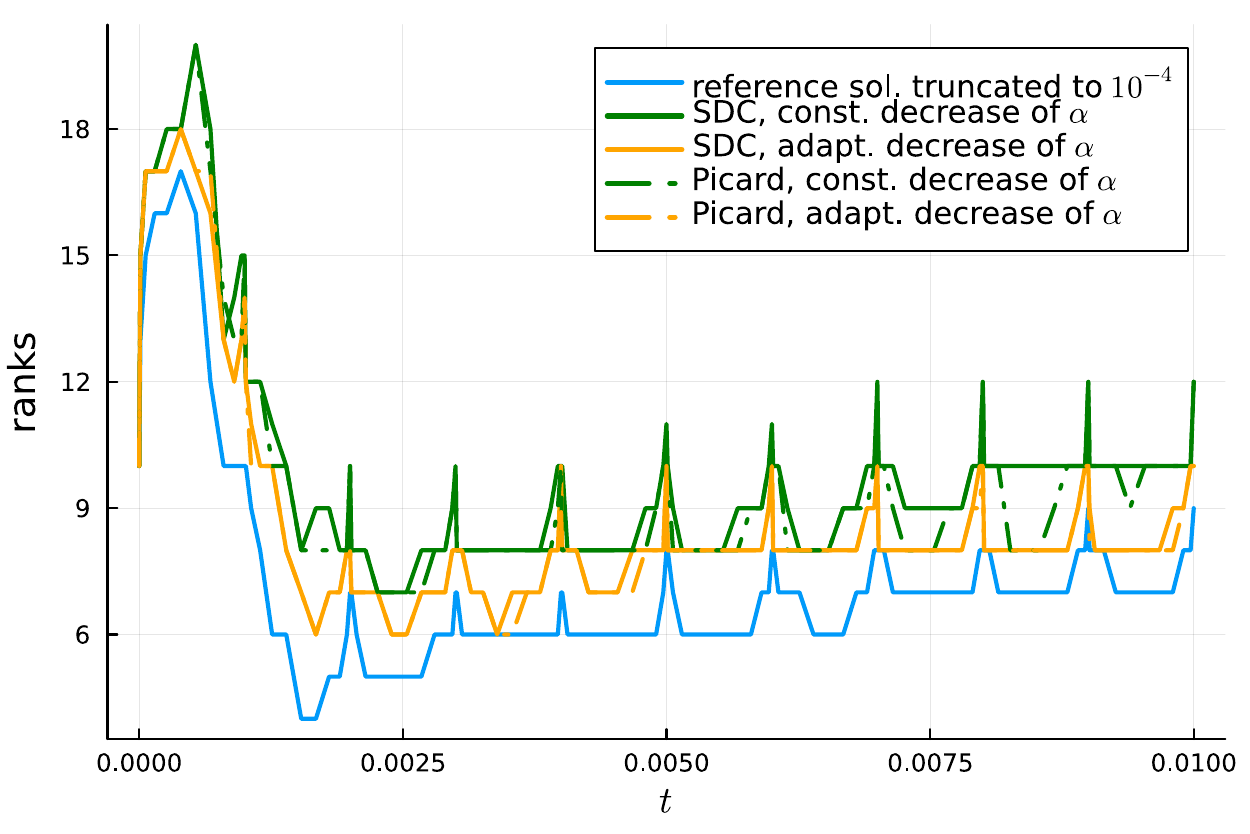}
	\end{minipage}%
	\begin{minipage}{0.5\textwidth}
		\centering
		\includegraphics[width=\textwidth]{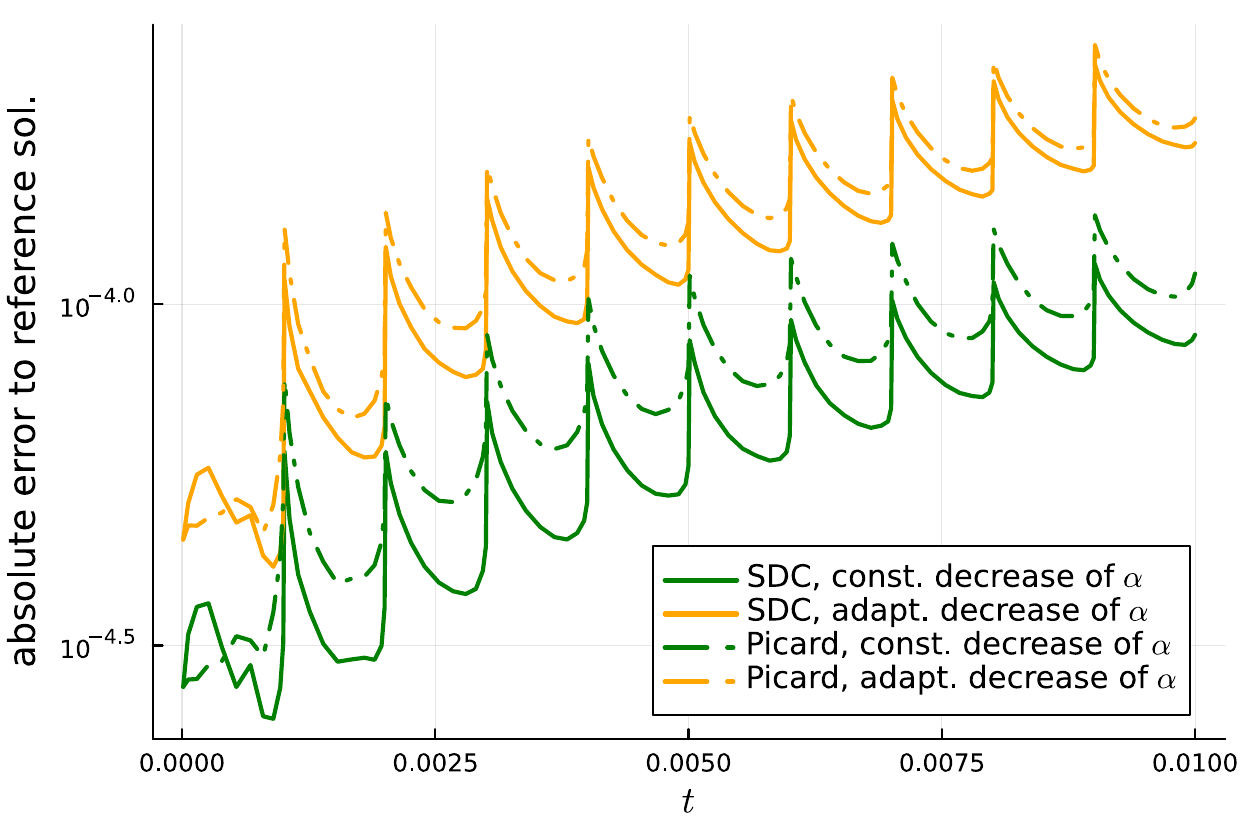}
	\end{minipage}
	\caption{Left: evolution of the ranks using different fixed-point iterations and strategies for decreasing of $\alpha$, with an iteration tolerance of $10^{-4}$, compared to the optimal ranks of a truncated reference solution. Right: corresponding evolution of the absolute errors with respect to the reference solution.}
	\label{fig:parabolic}
\end{figure}

The quasi-optimality of the proposed schemes is highlighted in Figure \ref{fig:parabolic}: here, we compare the resulting ranks and accuracies using different iterative solvers, which are iterated up to an accuracy of $10^{-4}$. The parameter $\theta$ is chosen again such that the numbers of needed iterations per interval are comparable, that is, $\theta = 0.5$ in the case of a constant decrease of the soft threshold and $\theta=0.2$ in the adaptive setting, in which $\alpha$ is decreased whenever
\begin{equation*}
	\bigl\| \mathbf u^{k+1} - \mathbf u^k \bigr\|_J \leq c \bigl\| \mathbf{u}^k - \mathbf{\Phi}(\mathbf{u}^k) \bigr\|_J
\end{equation*} 
is satisfied with $c = \frac25$. Note that $c$ was chosen such that the iterations on an interval were carried out at least twice (in our example: exactly twice), before being reduced. It becomes evident that for all four possible combinations of iterative thresholding schemes the obtained approximation ranks exhibit a quasi-optimal behavior by following, up to a constant only slightly larger than one, the trajectory of the optimal ranks, which are obtained by hard thresholding the full-rank reference solution to a tolerance of $10^{-4}$.
We also observe a linear increase in the absolute error over time, as one would expect from the analysis. 
Note that the errors are larger when $\alpha$ is decreased adaptively, or when using the Picard iteration, which is compensated by ranks slightly closer to optimal for these methods.

Regarding these numerical experiments, it is not completely clear how much of an advantage the SDC scheme yields over the Picard iteration, since the resulting error seems dominated by the value of the soft thresholding parameter. This observation also explains the much smaller absolute error in the first interval: as the ranks are higher, the fixed-point iteration reaches a smaller value of $\alpha$, as suggested by \eqref{error_dominated_by_alpha}, leading to a better accuracy.

\begin{figure}[t] 
	\centering
	\includegraphics[width=0.5\textwidth]{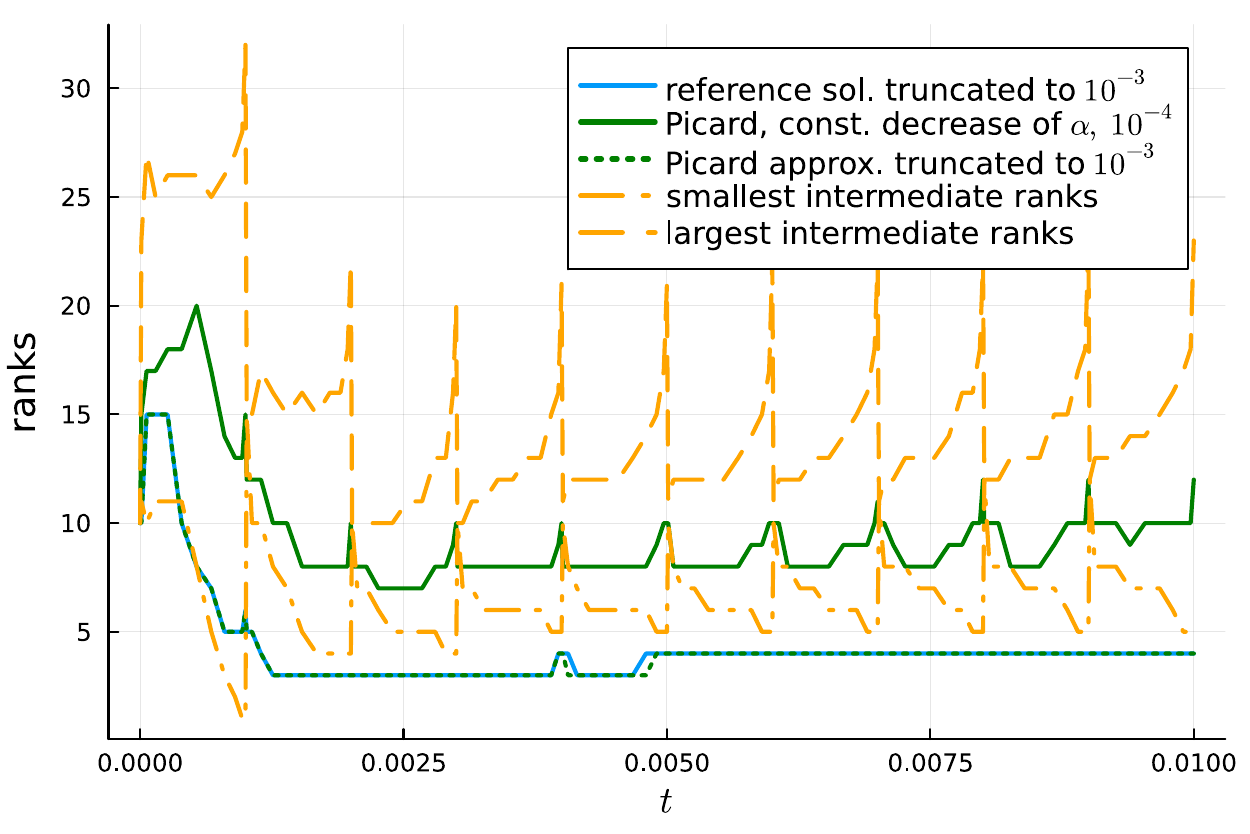}
	\caption{Comparison of the ranks of the computed approximation using Picard iterations and a constant decrease of $\alpha$ with $\theta=0.5$, its truncation to a tolerance of $10^{-3}$, the arising intermediate ranks in the computation of the approximation and the truncated reference solution.}
	\label{fig:parabolic_overview}
\end{figure}

In Figure \ref{fig:parabolic_overview}, we plot the minimal and maximal values of the ranks of all intermediate low-rank matrices appearing in the computation of the fixed-point iterations, at any of the times $t_j$. These ranks remain comparable to the ranks of the scheme, which only involves the last iteration in each interval, up to a factor about 2, corroborating the last statement of Theorem~\ref{thm:main}.
In particular, this shows that the secondary quadrature used in the approximation of the integrals over the exponentials does not cause an unbounded growth of the intermediate ranks.
The figure is obtained for the Picard iterative solver with truncation tolerance $10^{-4}$ and a constant decrease of $\alpha$; using instead an adaptive decrease of $\alpha$ produces almost identical results.
Moreover, the scheme
 is able to recover the optimal ranks almost exactly, when truncating to a slightly lower tolerance of $10^{-3}$ after the whole evolution has been computed.
At two points in time, we even observe smaller ranks for the truncated approximation, pointing out to the fact that the error of numerical integration does not necessarily result in an increase of the ranks.

\section{Conclusion}

In this work, we have proposed and analyzed a high-order time integration scheme that can be ensured to achieve the desired accuracy, while maintaining ranks that are controlled in terms of best approximation ranks. 
In particular, we avoid systematic errors that can occur in schemes based on tangent space projections. 
To the best of our knowledge, this is the first instance of a time integration scheme that achieves such a balance between accuracy and approximation ranks, aside from the method based on a space-time variational formulation in \cite{BF24}.

While the present method does not require such a variational formulation, in aiming for the iterative refinement of solutions on longer time intervals it can be regarded as a compromise between space-time methods and truncated time integration schemes.
As such, a question for future work is how to choose different -- and potentially, adaptive -- basis functions for the time variable.
Note, however, that the present choice of polynomials in time has certain advantages concerning, for example, the well-understood properties of mass and energy preservation.
There also remain some open questions regarding the analysis of the present high-order polynomial approximations, in particular the asymptotics of averaged Lebesgue constants noted in Remark \ref{rem:avglebesgue}.

We will address two generalizations as subjects for further work: in the context of Schr\"o\-din\-ger-type problems, applications to more general classes of potentials and the application to nonlinear problems, such as nonlinear Schr\"odinger equations; and more generally, to further types of different PDEs. In particular, a more detailed analysis of the more involved scheme for the parabolic case remains open at this point.
A second major direction of generalization concerns the extension of the method to high-dimensional problems via higher-order low-rank tensor representations. Although the basic scheme can be extended to this case directly based on results for hierarchical tensors in \cite{bachmayr_adaptive_2015,BS17}, the analysis becomes more subtle due to the need for tracking dependencies on the dimensionality.

\section*{Acknowledgements}

Co-funded by the European Union (ERC, COCOA, 101170147). Views and opinions expressed are however those of the authors only and do not necessarily reflect those of the European Union or the European Research Council. Neither the European Union nor the granting authority can be held responsible for them.

\bibliographystyle{plain}
\bibliography{BDSlowrankevolution.bib}

\section{Appendix}

In the following, we include some auxiliary lemmas needed in the proof of Theorem~\ref{thm:main}, and recall their proofs for the sake of completeness.

\begin{lemma}[Gauss-Legendre quadrature error]
	Let $\omega_j$ and $t_j$ for $1 \leq j \leq J$ be the Gauss-Legendre quadrature weights and nodes on the interval $[0,h]$. For any function $g$ in $ C^{2J}(0,h; L^2(\Omega))$, it holds that
	\begin{equation*}
		\Bigl\| \sum_{j=1}^J \omega_j g(t_j)-\int_0^h g(s) \, ds \Bigr\| \leq \frac{(J!)^4}{(2J+1)!(2J)!^2} \max_{s \in [0,h]} \bigl \|\partial_t^{2J}g(s) \bigr\|  h^{2J+1}.
	\end{equation*}
\label{quadrature_has_order_2J}
\end{lemma}

\begin{proof}
	Denote by $p\in \mathbb P_{2J-1}$ the polynomial that interpolates the function values $g(t_j)$ and the corresponding derivatives $g'(t_j)$ at all $J$ quadrature nodes $t_j$. As the Gauss-Legendre quadrature is exact for elements of $\mathbb P_{2J-1}$, we have
	\[
	\sum_{j=1}^J \omega_j g(t_j) = \sum_{j=1}^J \omega_j p(t_j) =\int_0^h p(s)\,ds,
	\]
	hence
	\begin{equation*}
	\Bigl\| \sum_{j=1}^J \omega_j g(t_j)-\int_0^h g(s) \, ds \Bigr\|
	\leq \int_0^h \|g(s)-p(s)\|\,ds.
	\end{equation*}
	Next, define the function $z(t) = \prod_{j=1}^J (t-t_j)$ and set, for some $s \in [0,h]\setminus\{ t_1,\dots,t_J\}$,
	\begin{equation*}
		h(t)=g(t)-p(t)-\frac{z(t)^2}{z(s)^2}\big(g(s)-p(s)\big).
	\end{equation*}
	As $h(x) = h(t_j) = \partial_t h(t_j) = 0$ for all $1\leq j\leq J$, the function $h$ has at least $2J+1$ zeros, counted with multiplicity, in the interval $[0,h]$.
	By a repeated application of Rolle's theorem, 
	its derivative of order $2J$, given by
	\begin{equation*}
		\partial_t^{2J} h(t)=\partial_t^{2J} g(t)-(2J)! \frac{g(s)-p(s)}{z(s)^2},
	\end{equation*}
	cancels at some point $\xi(s)\in [0,h]$. This implies that
	\begin{equation*}
		\| g(s)-p(s) \| = \biggnorm{ \frac{z(s)^2}{(2J)!}\partial_t^{2J} g(\xi(s))} \leq \frac{z(s)^2}{(2J)!} \max_{\xi \in [0,h]} \bigl\|\partial_t^{2J} g (\xi) \bigr\|.
	\end{equation*}
	Recalling that the $L^2(\Omega)$ normalized Legendre polynomial $L_J$ of degree $J$ on the reference interval $[-1,1]$ can be written as
	\begin{equation*}
	L_J(x)=\gamma_J\prod_{j=1}^J \Big(x-x_j\Big) = \gamma_Jz\Big(\frac{x+1}{2}h\Big),\quad \text{where} \quad \gamma_J=\sqrt\frac{2J+1}{2} \frac{(2J)!}{2^JJ!^2},
	\end{equation*}
	since the Gauss-Legendre nodes $t_j=\frac{x_j+1}{2}h$ are rescaled versions of the roots $x_j$ of $L_J$. We conclude with
	\[
	\int_0^h \frac{z(s)^2}{(2J)!}\,ds=\frac{h/2}{(2J)!}\int_{-1}^1 \bigg(\prod_{j=1}^J \frac{x-x_j}{2}h\bigg)^2dx = \frac{(h/2)^{2J+1}}{(2J)!\gamma_J^2}=\frac{(J!)^4h^{2J+1}}{(2J+1)!(2J)!^2}.\qedhere
	\]
\end{proof}

A similar, but weaker, result also holds within the interval $[0,h]$.

\begin{lemma}[Quadrature error inside the interval]
\label{lem_inner_quadrature}
	For $t\in [0,h]$ and $g\in C^{J}(0,h; L^2(\Omega))$ it holds that
	\begin{equation*}
		\biggnorm{ \int_0^t \Big(g(s) -\sum_{m=1}^J \ell_m(s)g(t_m)\Big)  \, ds } \leq \frac{J!}{\sqrt{2J+1}(2J)!} \max_{s \in [0,h]} \bigl \|\partial_t^Jg(s) \bigr\|  h^{J+1}.
	\end{equation*}
\end{lemma}

\begin{proof}
	Define again $z(t)=\prod_{j=1}^J (t-t_j)$, and let
	\[
	p=\sum_{m=1}^J g(t_m)\ell_m \in \mathbb P_{J-1}
	\]
	be the 
	Lagrange interpolant of $g$.
	For $s\neq t_1,\dots,t_J$, the function
	\[
	h:t \mapsto g(t)-p(t)-\frac{z(t)}{z(s)}\big(g(s)-p(s)\big)
	\]
	vanishes at $s$ and all nodes $t_1,\dots,t_J$, hence there exists some $\xi(s)\in [0,h]$ such that
	\[
	0=\partial_t^Jh(\xi(s))=\partial_t^Jg(\xi(s))-\frac{J!}{z(s)}\big(g(s)-p(s)\big),
	\]
	from which we can conclude that
	\[
	\|g(s)-p(s)\|\leq \frac{|z(s)|}{J!} \max_{\xi \in [0,h]} \|\partial_t^Jg(\xi) \|.
	\]
	We can therefore bound the quadrature error by
	\[
	\biggnorm{ \int_0^t \big(g(s) -p(s)\big)  \, ds }  \leq \int_0^t \|g(s)-p(s)\| \, ds \leq \int_0^h\frac{|z(s)|}{J!} \, ds\, \max_{\xi \in [0,h]} \|\partial_t^Jg(\xi) \|,
	\]
	and conclude the proof by a Cauchy-Schwarz inequality
	\[
	\biggl(\int_0^h\frac{|z(s)|}{J!^2} \, ds \biggr)^2 \leq \frac{h}{J!^2}\int_0^h z(s)^2 ds =\frac{h^{2J+2}}{2J+1}\frac{J!^2}{(2J)!^2}.\qedhere
	\]
\end{proof}

Finally, we need the following comparison between Frobenius distance and $\ell^2$ distance of singular values.
\begin{lemma}[Mirsky inequality] \label{mirsky}
Let $u,v\in L^2(\Omega_1\times \Omega_2)$, and denote their singular value decomposition as
\begin{equation*}
u = \sum_{k=1}^\infty \sigma_k(u)\, u^{(1)}_k \otimes u^{(2)}_k\quad\text{and}\quad  v = \sum_{k=1}^\infty \sigma_k(v)\, v^{(1)}_k \otimes v^{(2)}_k,
\end{equation*}
with the singular values $\sigma_k$ in non-increasing order. It then holds that
\[
\|u-v\|_{L^2(\Omega_1 \times \Omega_2)}^2\geq \sum_{k=1}^\infty |\sigma_k(u)-\sigma_k(v)|^2.
\]
\end{lemma}

This result was originally shown to hold in a finite-dimensional setting. Its extension to Hilbert-Schmidt operators, or equivalently to $L^2$ functions on a separable tensor product domain, can be found in \cite[Corollary 5.3]{ASMarkus_1964}. We include a self-contained proof below.

\begin{proof}
First, we restrict ourselves to hermitian operators, by considering eigendecompositions
\[
U:=\begin{pmatrix} 0& u\\u^*&0\end{pmatrix}=\sum_{k\in \Z\setminus\{0\}} \lambda_k\, U_kU_k^*\quad\text{and}\quad V:=\begin{pmatrix} 0& v\\v^*&0\end{pmatrix}=\sum_{k\in \Z\setminus\{0\}} \mu_k\, V_kV_k^*
\]
with
\[
\lambda_{\pm k}=\pm \sigma_k(u),\qquad U_{\pm k}=\frac{1}{\sqrt 2}\begin{pmatrix} u^{(1)}_k \\ \pm u^{(2)}_k\end{pmatrix},\qquad \mu_{\pm k}=\pm \sigma_k(v)\quad\text{and}\quad V_{\pm k}=\frac{1}{\sqrt 2}\begin{pmatrix} v^{(1)}_k \\ \pm v^{(2)}_k\end{pmatrix}.
\]
Using $(U_k)_{k \in \Z\setminus\{0\}}$ and $(V_\ell^*)_{\ell \in \Z\setminus\{0\}}$ as orthonormal bases for the columns and rows, we obtain in the Frobenius norm 
\[
\|U-V\|_{F}^2
=\bigg\|\sum_{k,\ell\in \Z\setminus\{0\}} (\lambda_k-\mu_\ell)\,U_k(U_k^*V_\ell)V_\ell^*\bigg\|_F^2
=\sum_{k,\ell\in \Z\setminus\{0\}} p_{k,\ell} (\lambda_k-\mu_\ell )^2,
\]
where  $p_{k,\ell}=|U_k^*V_\ell|^2$.
The matrix $(p_{k,\ell})_{k, \ell \in \Z\setminus\{0\}}$ is bistochastic, that is, it has non-negative entries satisfying
\[
\sum_{\ell\in \Z\setminus\{0\}} p_{k,\ell}=U_k^*\!\sum_{\ell\in \Z\setminus\{0\}} V_\ell V_\ell^*\; U_k=\|U_k\|^2=1, \quad \text{for all} \quad k \in \Z\setminus\{0\},
\]
and
\[
\sum_{k\in \Z\setminus\{0\}} p_{k,\ell}=V_\ell^*\!\sum_{k\in \Z\setminus\{0\}} U_kU_k^* \;V_\ell=\|V_\ell\|^2=1, \quad \text{for all} \quad \ell \in \Z\setminus\{0\}.
\]
Therefore, it is in the closed convex hull of permutation matrices $p^\sigma=(\mathds 1_{\ell=\sigma(k)})_{k,\ell}$, with $\sigma$ a permutation of $\Z\setminus\{0\}$.
For such matrices, a simple reordering inequality shows that
\[
\sum_{k,\ell\in \Z\setminus\{0\}} p^\sigma_{k,\ell}(\lambda_k-\mu_\ell)^2=\sum_{k\in \Z\setminus\{0\}} (\lambda_k-\mu_{\sigma(k)})^2=\sum_{k\in \Z\setminus\{0\}}\lambda_k^2-2\lambda_k\mu_{\sigma(k)}+\mu_{\sigma(k)}^2\geq \sum_{k\in \Z\setminus\{0\}} (\lambda_k-\mu_{k})^2.
\]
Taking convex combinations, this remains true with $p$ instead of $p^\sigma$, and we conclude the inequality with
\begin{equation*}
	2\|u-v\|_{L^2(\Omega_1 \times \Omega_2)}^2=\|U-V\|_{F}^2\geq \sum_{k\in \Z\setminus\{0\}} (\lambda_k-\mu_{k})^2=2\sum_{k=1}^\infty |\sigma_k(u)-\sigma_k(v)|^2. \qedhere
\end{equation*}
\end{proof}

\end{document}